\documentclass[11pt]{amsart}
 \usepackage{amsmath,amssymb}
 \usepackage{xcolor}

\usepackage{tikz}
\usepackage{empheq}
\usepackage{tikz-cd}

\newtheorem{theorem}{Theorem}[section]
\newtheorem{thm}[theorem]{Theorem}
\newtheorem{lemma}[theorem]{Lemma}

\newtheorem{prop}[theorem]{Proposition}

\newtheorem{corollary}[theorem]{Corollary}

\theoremstyle{definition}
\newtheorem{definition}[theorem]{Definition}
\newtheorem{defi}[theorem]{Definition}

\newtheorem{remark}[theorem]{Remark}

\theoremstyle{plain}
\newtheorem*{namedthm}{\namedthmname}
\newcounter{namedthm}

\makeatletter

 \newcommand{\R}{\mathbb R}
 \newcommand{\Q}{\mathbb Q}
 \newcommand{\C}{\mathbb C}
  
 \newcommand{\N}{\mathbb N}

 \newcommand{\vep}{\varepsilon}
 
 \newcommand{\capa}{{\rm Cap}}

 \newcommand \psh {{\rm PSH}}
 \newcommand \PSH {{\rm PSH}}

\newcommand{\id}{{\bf 1}}
 
 \newcommand \Amp {{\rm Amp}}
 \newcommand \rmE {{\rm E}}
 \newcommand \vol{{\rm Vol}}
 \newcommand \Ent{{\rm Ent}}
\newcommand{\Ric}{{\rm Ric}}
 \newcommand \cE{\mathcal E }

 \newcommand \mN{\mathcal N}

 \newcommand \cM {\mathcal{M}}

\definecolor{ao}{rgb}{0.0, 0.5, 0.0}

 \numberwithin{equation}{section}

 \usepackage{hyperref}
\hypersetup{
    unicode=false,        
    pdftoolbar=true,      
    pdfmenubar=true,       
    pdffitwindow=false,     
    pdfstartview={FitH},    
    pdftitle={Convexity of the Mabuchi functional in the big setting},    
    pdfauthor={Di Nezza, Trapani, Trusiani},     
    colorlinks=true,       
   linkcolor=blue,          
    citecolor=blue,        
    filecolor=black,      
    urlcolor=blue}

\usepackage{geometry}


\begin{document}

\title[Convexity of the Mabuchi functional in the big setting]{Convexity of the Mabuchi functional in big cohomology classes}
\author{Eleonora Di Nezza, Stefano Trapani, Antonio Trusiani}

\date{\today}

\begin{abstract}
 We study the Mabuchi functional associated to a big cohomology class. We define an invariant associated to transcendental Fujita approximations, whose vanishing is related to the Yau-Tian Donaldson conjecture. Assuming vanishing (finiteness) of this invariant we establish (almost) convexity along weak geodesics. As an application, we give an explicit expression of the distance $d_p$ in the big setting for finite entropy potentials.
\end{abstract}

 \maketitle

\section{Introduction}
Let $X$ be a compact K\"{a}hler manifold of complex dimension $n$ and fix a K\"{a}hler form $\omega$. Let $d$ and $d^c$ be the real differential operators defined as $d:=\partial +\bar{\partial},\,  d^c:=\frac{i}{2\pi}\left(\bar{\partial}-\partial \right).$ 
  By the $d d^c$-lemma, the space 
 of K\"{a}hler forms cohomologus to $\omega$ can be identified with the space 
 $$ { \mathcal{H}}  =  \{ u \in \mathcal{C}^{\infty}(X) : \omega + dd^c u > 0 \} / \R.$$
To study canonical K\"{a}hler metrics on $X$, Mabuchi in \cite{Mab86}, \cite{Mab87} introduced a natural  Riemannian metric $g$ on $\mathcal{H}.$ He defined the Mabuchi functional $\mathcal{M}$ (known as well as K-energy) such that its critical points are constant scalar curvature K\"ahler (cscK for short) metrics. Furthermore, he demonstrated that the Mabuchi functional is convex along smooth geodesics of $(\mathcal{H}, g)$.  However, $(\mathcal{H}, g)$ is an infinite dimensional Fr\'echet Riemannian manifold, hence the existence of smooth geodesics is not guaranteed, as shown in \cite{LV13}, \cite{DL12}. Nevertheless, a natural notion of weak geodesics exists to connect two points in $\mathcal{H}$. In \cite{BB17}, Berman and Berndtsson proved convexity of the Mabuchi functional along such weak geodesics and, as a consequence they established the uniqueness of the cscK metric in a given cohomology class (whenever it exists).

 \smallskip

In the 1990's Tian  \cite{Tian94} made an influential conjecture stating that the existence of a cscK metric is equivalent to the properness of the Mabuchi functional. \\
There were several attempts by many in this direction. The conjecture was first proven in the (Fano) K\"ahler-Einstein case by Darvas and Rubinstein \cite{DR17}. The fact that the existence of a cscK metric implies the properness of the $K$-energy is due to Berman, Darvas and Lu \cite{BDL20}, while the reverse implication was proven more recently by Chen and Cheng \cite{ChCh1, ChCh2} (see also \cite{LTW21}, \cite{Li22}, \cite{PTT23}, \cite{PT24} for some results in the singular setting). 

Motivated by the classification problem in birational geometry, in \cite{BBGZ13}, the authors studied the K\"ahler-Einstein equation as a solution to a similar variational problem, but in the more singular context of a big cohomology class $\{\theta\}\in H^{1,1} (X, \R)$. We say that $\{\theta\}$ is big if and only if its volume $\vol(\theta)>0$ and we define this notion in Section \ref{sec:prelim}. 
The notion of big cohomology classes is in fact invariant under bimeromorphic maps, while this is not the case for K\"ahler classes, and naturally arises in algebraic geometry.

  \smallskip

 In this paper we thus define and study the (relative) Mabuchi functional in a big cohomology class $\{\theta\}$. We recall that a function $\varphi: X\rightarrow \mathbb{R}\cup\{-\infty\}$ is quasi-plurisubharmonic (qpsh) if it can be locally written as the sum of a plurisubharmonic function and a smooth function. $\varphi$ is called $\theta$-plurisubharmonic  ($\theta$\emph{-psh}) if it is qpsh and $\theta_\varphi:=\theta+dd^c \varphi\geq 0$ in the sense of currents. We let $\psh(X,\theta)$ denote the set of $\theta$-psh functions that are not identically $-\infty$. Thanks to \cite{BEGZ10}, given a $\theta$-psh function $\varphi$, the \emph{non-pluripolar Monge-Amp{\`e}re measure} of $\varphi$ is well defined and denoted by 
$\theta_\varphi^n:=(\theta+dd^c\varphi)^n.$ In $\psh(X,\theta)$ there exists a ``best candidate" which is less singular than any other $\theta$-psh function:
 $$V_\theta:=\sup\{ u\in \psh(X,\theta), \, u\leq 0 \}.$$
 We then say that a $\theta$-psh function $\varphi$ has minimal singularities if it is relative bounded with respect to $V_\theta$, i.e. $|V_\theta-\varphi|\leq C$, for some $C>0$.\\
Also, we have that $\vol(\theta)= \int_X \theta_{V_\theta}^n$. Thus, the fact that $\{\theta\}$ is big means that the mass of (the Monge-Amp{\`e}re measure associated to) $V_\theta$ is strictly positive.
The function $V_\theta$ is then the less singular function with ``full mass". \\
For any other mass $0<m< \vol(\theta)$, one can consider \emph{model potentials} $\phi$ having that mass, $\int_X \theta_\phi^n=m$, which are going to play the role of $V_\theta$ in a relative setting. We refer to Section \ref{subsec env} for a precise definition.

 \smallskip
 
 Taking inspiration from the Kähler setting (and from the big and nef case studied by Di Nezza and Lu in \cite{DNL21}), we define the Mabuchi functional relative to $(X,\theta_\varphi)$, for any closed and positive $(1,1)$-current $\theta_\varphi$ with \emph{well defined Ricci curvature} (see Section \ref{sec: convex}), as
\begin{equation}\label{eqM intro}
 	\cM_{\theta, \varphi}(u) := \bar{S}_{\varphi} E(\theta; u,{\varphi}) - n E_{\Ric (\theta_\varphi)} (\theta; u,{\varphi}) +\Ent(u,\varphi),
\end{equation}
for any $u$ with the same singularity type of $\varphi$, that is $|u-\varphi|\leq C$, for some $C>0$. We refer to Sections \ref{sec: entr} and \ref{sec: convex} for the definitions of the energy terms $E, E_{\Ric(\theta_\varphi)}$ and the entropy term $\Ent$ appearing in \eqref{eqM intro}.

The main result of the paper is the (almost)-convexity of $\cM_{\theta, \varphi}$ along weak geodesic. Our strategy is to treat this problem as a limiting case of classes with ``prescribed singularities". The latter notion was introduced in \cite{DDL1}, \cite{DDL5}, \cite{DDL6}  and further developed in \cite{Trus19}, \cite{Trus20}, \cite{Trus20b}. 

To be more precise, we consider a \emph{monotone transcendental Fujita approximation} of $\{\theta\}$, i.e. a sequence of model 
potentials $(\phi_k)_k\subset \PSH(X,\theta) $, $\phi_k\nearrow V_\theta$ such that for any $k$, there exists a modification $\pi_k : Y_k \rightarrow X$, $Y_k$ compact K\"ahler manifold, and
$$ \pi_k^* (\theta+dd^c {\phi_k}) =(\eta_k+ dd^c{\tilde{\phi}_k}) +[F_k] $$
where $F_k$ is an effective $\R$-divisor, $\tilde{\phi}_k$ is a potential with minimal singularities and $\eta_k $ represents a big and nef class. The existence of a sequence $(\phi_k)_k$ is a consequence of \cite{Dem92} (see Lemma \ref{lem:Fujita}). We refer to Section \ref{sec:fuj} for more details.\\
There is a key quantity associated to a given monotone transcendental Fujita approximation $(\phi_k)_k$:
$$H(\phi_k):= \liminf_{k\to +\infty} \{\eta_k^{n-1}\}\cdot K_{Y_k/X}.$$
Note that $H(\phi_k)\geq 0$ and let us stress that it does not depend on the modifications $\pi_k$ (Lemma \ref{lem:Independence on modification}). We then consider $$H:= \inf\{ H(\phi_k), \, (\phi_k)_k \, {\rm\, monotone\; transcendental\;  Fujita \;approximation}\}.$$ 

Our main result states as follows:

 \begin{thm}
   Let $u_0,u_1\in\psh(X,\theta)$ with minimal singularities and let $(u_t)_{t\in [0,1]}$ be the weak geodesic connecting $u_0$ and $u_1$. Let $\varphi\in \cE(X,\theta)$ be such that $\theta_\varphi^n=\vol(\theta)\omega^n$, $\sup_X \varphi=0$. Then $u_t$ has minimal singularities and the function  $t\mapsto \cM_{\theta, \varphi}(u_t)$ is almost convex in $[0,1]$, i.e.
\begin{equation}\label{intro:Almost_Convexity_Statement_Thm}
    \cM_{\theta,\varphi}(u_t)\leq (1-t)\cM_{\theta,\varphi}(u_0)+t\cM_{\theta,\varphi}(u_1)+ \frac{n\lVert u_0-u_1\rVert_\infty }{2\vol(\theta)} H.
\end{equation}
\end{thm}
The proof consists in two big steps: first, for any monotone trascendental Fujita approximation $(\phi_k)_k$ and for each $k$, we prove (almost)-convexity for $ \cM_{\theta,\varphi_k} $ where $\varphi_k$ is a suitable $\theta$-psh function such that $|\varphi_k-\phi_k|\leq C$; then we perform a limiting procedure.
 \smallskip
 
Also the inequality in \eqref{intro:Almost_Convexity_Statement_Thm} is interesting only if the quantity $H$ is finite or equal to zero.

\noindent As a consequence of our theorem we get the following:

\begin{corollary}
  Assume $H=0$. Then the function $t\mapsto \cM_{\theta, \varphi}(u_t)$ is convex in $[0,1]$.
\end{corollary}
When $\{\theta\}$ is big and nef we find $H=0$. The convexity of the Mabuchi functional in the big and nef case was proved by Di Nezza and Lu in \cite{DNL21} using the fact that a big and nef class can be approximated by K\"ahler classes.
 

\smallskip
In Section \ref{sec:YTD} we examine the condition $H=0$. Notably, we observe that $H=0$ when $\{\theta\}$ has a bimeromorphic Zariski decomposition (Theorem \ref{thm:We_have_Convexity}). From this, it follows easily that $\cM_{\theta,\varphi}$ is convex in complex dimension 2.
\smallskip

As it was observed in \cite{Li23}, when $X$ is a projective manifold, and $\alpha$ is the cohomology class of a big $\Q$-divisor, if $H=0$ then one can solve the Yau-Tian-Donaldson Conjecture (see Remark \ref{rk:B YTD}). It is conjectured that $H$ is zero \cite[Conjecture 4.7]{Li23}. This raises the natural question of how the convexity of the Mabuchi functional for any big integral class is connected to the Yau-Tian-Donaldson Conjecture.

\medskip 

Once the convexity of the Mabuchi functional is established, we manage to ensure a uniform control of the entropy along the geodesic segment:

\begin{corollary}\label{cor: bound entropy geo}
   Assume $H<+\infty$. Let $C_1>0$ and let $u_0,u_1\in \PSH(X,\theta)$ be such that $u_0-u_1$ is bounded. Assume $\Ent(u_0),\Ent(u_1)\leq C_1$. Then there exists a positive constant $C_2$ such that
    $$
    \Ent(u_t)\leq C_2
    $$
    for any $t\in [0,1]$, where $C_2$ only depends on $C_1, n,X, \{\omega\},\{\theta\}, \lVert u_0-u_1\rVert_\infty$, $H$ and on a lower bound of $\vol(\theta)$.
\end{corollary}

This observation is the key to show that the distance $d_1$ admits an explicit expression. Let us recall that, following \cite{DDL3}, the distance $d_1$ (in the big setting) is defined as
$$d_1(u_0, u_1):= E(\theta;u_0, V_\theta)+E(\theta;u_1, V_\theta) -2 E(\theta; P_\theta(u_0, u_1), V_\theta),$$
where $E$ denotes the energy functional as in \eqref{eqM intro} and 
$$P_\theta(u_0, u_1):= \sup \{v\in \psh(X, \theta), \, v\leq \min(u_0, u_1)\}.$$

\begin{thm}\label{thm dist}
Assume $H<+\infty$. Let $u_0,u_1\in \Ent(X,\theta)$ and $u_t$ be the weak geodesic connecting $u_0$ and $u_1$. If $u_0-u_1$ is bounded then
    \begin{equation}\label{eqn: intro di}
        d_1(u_0,u_1)=\int_X | \dot{u}_t| \,\theta_{u_t}^n
    \end{equation}
    for any $t\in [0,1]$.
\end{thm}

For the purpose of the introduction, we state the above theorem for $d_1$ since the latter can be more easily defined in the big setting. But we actually can prove the analogous result for all the Finsler distances $d_p$, $p\geq 1$, recently defined in the big setting by Gupta \cite{Gup23b} (see Section \ref{sec geo dis} and Theorem \ref{thm dist big}).

\smallskip
Let us recall that \eqref{eqn: intro di} was first proved by Chen \cite{Chen00} for $d_2$ in the K\"ahler case, specifically for smooth K\"ahler potentials. This equality was instrumental in demonstrating that $d_2$ is a genuine distance, rather than merely a semi-distance.

The distances $d_p$ on ``extended Mabuchi spaces", the so called Monge-Amp\`ere energies classes $\mathcal{E}^p$, were initially introduced by Darvas \cite{Dar17AJM} in the K\"ahler setting. Since then, they have been extensively studied by various authors due to their crucial role in the variational approach to finding K\"ahler-Einstein or cscK metrics. These distances have been defined in singular contexts through approximation procedures \cite {DNG18}, \cite{DDL3}, \cite{DNL20}, \cite{Xia19b}, \cite{Gup23b}. 

Having an explicit formulation for $d_p$ is crucial for several applications. For example, it enables us to extend the result of \cite{DNT21} and \cite{DNL21} on Monge-Amp\`ere measures on contact sets (see Proposition \ref{thm: MA contact big} and Corollary \ref{thm: domination}). Knowing the support of the Monge-Amp\`ere measure of a (singular) $\theta$-psh function is indeed a cornerstone of pluripotential theory.

 \medskip
 \noindent {\bf Acknowledgement.} The first author is supported by the project SiGMA ANR-22-ERCS-0004-02. \\
 Part of this material is based upon work done while the first author was supported by the National Science Foundation under Grant No. DMS-1928930, while the author was in residence at the Simons Laufer Mathematical Sciences Institute
(formerly MSRI) in Berkeley, California, during the Fall 2024 semester.
 
  \smallskip
 The second author is partially supported by PRIN \emph{Real and Complex Ma\-ni\-folds: Topology, Geometry and holomorphic dynamics} n.2017JZ2SW5, and by MIUR Excellence Department Projects awarded to the Department of Mathematics, University of Rome Tor Vergata, 2018-2022 CUP E83C18000100006, and 2023-2027. CUP E83C23000330006.
 \smallskip
 
 The third author is partially supported by the Knut and Alice Wallenberg Foundation.
 \smallskip

\section{Preliminaries}\label{sec:prelim}

We recall results from (relative) pluripotential theory of big cohomology classes. We borrow notation and terminology from \cite{DDL6}.

Let $(X,\omega)$ be a compact K\"ahler manifold of dimension $n$. Let $\theta$ be a smooth closed $(1,1)$-form on $X$. A function $\varphi: X\rightarrow \mathbb{R}\cup\{-\infty\}$ is quasi-plurisubharmonic (qpsh)  if it can be locally written as the sum of a plurisubharmonic function and a smooth function. $\varphi$ is called $\theta$-plurisubharmonic  ($\theta$\emph{-psh}) if it is qpsh and $\theta_\varphi:=\theta+dd^c \varphi\geq 0$ in the sense of currents. We let $\psh(X,\theta)$ denote the set of $\theta$-psh functions that are not identically $-\infty$.
In the whole paper we will assume that $\{\theta\}$ is \emph{big}, which means that it admits a \emph{Kähler current}, i.e. there exists $\psi \in {\rm PSH}(X,\theta)$ such that $\theta +dd^c \psi \geq \vep \omega$ for some small constant $\vep>0$. Here, $d$ and $d^c$ are real differential operators defined as $d:=\partial +\bar{\partial},\,  d^c:=\frac{i}{2\pi}\left(\bar{\partial}-\partial \right).$

We say that a $\theta$-psh function $\varphi$ has 
\emph{analytic singularities} if there exists a constant $c>0$ such that locally on $X$,
\begin{equation}
    \label{eqn:An_Sing}
    \varphi={c}\log\sum_{j=1}^{N}|f_j|^2+g,
\end{equation}
where $g$ is bounded and $f_1,\dots,f_N$ are local holomorphic functions.

\medskip

Demailly regularization's theorem ensures that there are plenty of K\"ahler currents with analytic singularities (see for e.g. \cite[Theorem 3.2]{DP04}).

The \emph{ample locus} $\Amp(\theta)$ of $\{ \theta\}$ is the set of points $x\in X$ such that there exists a K\"ahler current $T\in \{ \theta \}$ with analytic singularity smooth in a neighbourhood of $x$. The ample locus $\Amp({\theta})$ is a Zariski open subset, and it is nonempty \cite{Bou04}. The complement of the ample locus is known as the \emph{non-Kähler locus}, $\rmE_{nK}({\theta})$.

If $\varphi$ and $\varphi'$ are two $\theta$-psh functions on $X$, then $\varphi'$ is said to be \emph{less singular} than $\varphi$, i.e. $\varphi \preceq \varphi'$, if they satisfy $\varphi\le\varphi'+C$ for some $C\in \mathbb{R}$.
We say that $\varphi$ has the same singularity as $\varphi'$, i.e. $\varphi \simeq \varphi'$, if $\varphi \preceq \varphi'$ and $\varphi' \preceq \varphi$. The latter condition is easily seen to yield an equivalence relation, whose equivalence classes are denoted by $[\varphi]$, $\varphi \in \psh(X, \theta)$.

A $\theta$-psh function $\varphi$ is said to have \emph{minimal singularity type} if it is less singular than any other $\theta$-psh function.  Such $\theta$-psh functions with minimal singularity type always exist, one can consider for example
\begin{equation*}
V_\theta:=\sup\left\{ \varphi\,\,\theta\text{-psh}, \varphi\le 0\text{ on } X \right \}. 
\end{equation*}
Trivially, a $\theta$-psh function with minimal singularity type is locally bounded  in $\Amp({\theta})$.  It follows from \cite[Theorem 1.1]{DNT23} that $V_{\theta}$ is $C^{1, \bar{1}}$ in the ample locus ${\rm Amp}({\theta})$.

Given $\theta^1+dd^c\varphi_1,..., $ $ \theta^p+dd^c \varphi_p$   positive $(1,1)$-currents, where $\theta^j$ are closed smooth real $(1,1)$-forms, following the construction of Bedford-Taylor \cite{BT87} in the local setting, it has been shown in \cite{BEGZ10} that the sequence of currents 
\[
{\bf 1}_{\bigcap_j\{\varphi_j>V_{\theta_j}-k\}}(\theta^1+dd^c \max(\varphi_1, V_{\theta_1}-k))\wedge...\wedge (\theta^p+dd^c\max(\varphi_p, V_{\theta_p}-k))
\] 
is non-decreasing in $k$ and converges weakly to the so called \emph{non-pluripolar product} 
\begin{equation}\label{eq: BEGZ10_def}
\langle \theta^1_{\varphi_1 } \wedge\ldots\wedge\theta^p_{\varphi_p}\rangle .
\end{equation}
In the following, with a slight abuse of notation, we will denote the non-pluripolar product simply by $\theta^1_{\varphi_1 } \wedge\ldots\wedge\theta^p_{\varphi_p}$.
When $p =n$, the resulting positive $(n,n)$-current is a Borel measure that does not charge pluripolar sets. Pluripolar sets are Borel measurable sets that are contained in some set $\{\psi = -\infty\}$ (as it follows from \cite[Corollary 2.11]{BBGZ13}).

For a $\theta$-psh function $\varphi$, the \emph{non-pluripolar complex Monge-Amp{\`e}re measure} of $\varphi$ is
$$
\theta_\varphi^n:=(\theta+dd^c\varphi)^n.
$$

The volume of a big class $\{ \theta\}$  is defined by 
\[
{\rm Vol}({\theta}):= \int_{{\rm Amp}(\{\theta\})} \theta_{V_\theta}^n.
\]
For notational convenience in the following we simply write ${\rm Vol}(\theta)$, but keeping in mind that the volume is a cohomological constant.\\
By \cite[Theorem 1.16]{BEGZ10}, in the above expression one can replace $V_\theta$ with any $\theta$-psh function with minimal singularity type. A $\theta$-psh function $\varphi$ is said to have \emph{full Monge--Amp\`ere mass} if
\[
\int_X \theta_\varphi^n={\rm Vol}(\theta),
\]
and we then write $\varphi\in \mathcal{E}(X,\theta)$.

An important property of the non-pluripolar product is that it is local with respect to the plurifine topology (see  \cite[Corollary 4.3]{BT87},\cite[Section 1.2]{BEGZ10}).  
This topology is the coarsest such that all qpsh  functions on $X$ are continuous. For convenience we record the following version of this result for later use. 
\begin{lemma} \label{lem: plurifine}
Fix closed smooth  real big $(1,1)$-forms $\theta^1,...,\theta^n$.  Assume that $\varphi_j,\psi_j,j=1,...,n$ are $\theta^j$-psh functions such that $\varphi_j =\psi_j$ on an open set $U$ in the plurifine topology. Then 
$$
\id_{U} \theta^1_{\varphi_1} \wedge ... \wedge \theta^n_{\varphi_n} = \id_{U} \theta^1_{\psi_1} \wedge ... \wedge \theta^n_{\psi_n}.
$$
\end{lemma}
Lemma \ref{lem: plurifine} will be referred to as the \emph{plurifine locality property}. We will often work with sets of the form $\{u<v\}$, where $u,v$ are quasi-psh functions. These are always open in the plurifine topology. 
\medskip

\noindent The classical Monge-Amp\`ere capacity (see \cite{BT82}, \cite{Kol98}, \cite{GZ05}) is defined by 
\[
\capa_{\omega}(E) := \sup \left \{ \int_E (\omega+dd^c u)^n \; : \; u\in \PSH(X,\omega),\; -1\leq u \leq 0 \right \}. 
\]
A sequence $u^k$ converges in capacity to $u$ if for any $\varepsilon>0$ we have
\[
\lim_{k\to +\infty} \capa_{\omega}(\{|u^k-u|\geq \varepsilon\}) =0. 
\]

\medskip
\noindent {The following extension of \cite[Theorem 2.6]{DDL6} will be used several times in the paper.}

  \begin{theorem}\label{thm:weak convergence}
    For $j\in \{1,\ldots,n\}$, let $\{\theta^j_k\}_k$ be a sequence of smooth closed real $(1,1)$-forms smoothly converging to smooth forms $\theta^j$ representing big cohomology classes. Suppose that for all  $j \in \{1,\ldots,n\}$ we  have  $u_j\in \PSH(X,\theta),u_j^k \in \PSH(X,\theta^j_k)$ such that $u_j^k \to u_j$ in capacity as $k \to + \infty.$ Let $\chi_k,\chi \geq 0$  be quasi continuous and uniformly bounded functions such that $\chi_k \to \chi$ in capacity. Then
    \begin{equation}
        \label{eqn:Lsc of MA operator}
        \liminf_{k \to +\infty} \int_X \chi_k \theta^1_{{k,}u_1^k} \wedge 
\wedge \theta^2_{{k,}u_2^k} \wedge \ldots  \wedge \theta^n_{{k,}u_n^k} \geq \int_X \chi \theta^1_{u_1} \wedge \theta^2_{u_2} \wedge \ldots  \wedge \theta^n_{u_n}.
    \end{equation}
In addition, if
\begin{equation}
    \label{eqn:Usc of MA masses}
    \limsup_{k \to +\infty} \int_X  \theta^1_{{k,}u_1^k} \wedge \theta^2_{{k,}u_2^k} \wedge \ldots  \wedge \theta^n_{{k,}u_n^k} \leq  \int_X  \theta^1_{u_1} \wedge  \theta^2_{u_2} \wedge \ldots  \wedge \theta^n_{u_n}
\end{equation}
then $$\chi_k  \theta^1_{{k,}u_1^k} \wedge \theta^2_{{k,}u_2^k} \wedge \ldots  \wedge \theta^n_{{k,}u_n^k} \to  \chi \theta^1_{u_1} \wedge \wedge \theta^2_{u_2} \wedge \ldots  \wedge \theta^n_{u_n} $$
in the weak sense of measure on $X.$ 
\end{theorem}
\begin{proof}
Fix $j$ and let $T_j=\theta^j+dd^c\varphi_j$ be a Kähler current in $\{\theta^j\}$ with analytic singularities along the non-Kähler locus of $\{\theta_j\}$ (such a $T_j$ exists thanks to \cite[Theorem 3.17]{Bou04}). Let $\varepsilon_j>0$ such that $T_j\geq 2\varepsilon_j\omega$. For $k>>1$ we have $\theta^j_k\geq \theta^j-\varepsilon_j\omega$. In particular, $\varphi_j\in \PSH(X,\theta_k^j)$ and $\theta_k^j+dd^c \varphi_j$ is a Kähler current with analytic singularities. It then follows that for $k>>1$ we have $E_{nK}({\theta^j_k})\subseteq E_{nK}({\theta^j})$. Thus (up to take $k$ big enough) we can assume that for any $k$,
    \begin{equation}\label{amp}
    \Amp({\theta^j_k})\supseteq \Omega:=\cap_{j=1}^n \Amp({\theta^j}).
    \end{equation}
    We then claim that $V_{\theta_k^j}$ converges to $V_{\theta^j}$ in capacity. It follows from the smooth convergence $\theta_k^j\to \theta^j$ that for any $\varepsilon>0$ there exists $k_0=k_0(\varepsilon)>>1$ such that $\theta^j -\varepsilon \omega \leq \theta^j_k \leq \theta^j -\varepsilon \omega$ for all $k\geq k_0$. It then follows by definition that for all $k\geq k_0$,
    $$
    V_{\theta^j-\varepsilon\omega}\leq V_{\theta_k^j}\leq V_{\theta^j+\varepsilon\omega}.
    $$
     Thus to conclude the claim it is enough to show that $V_{\theta^j-\varepsilon\omega}\nearrow V_{\theta^j}$ and $V_{\theta^j+\varepsilon\omega}\searrow V_{\theta^j}$ as $\varepsilon\searrow 0$. Observe that $\{V_{\theta^j-\varepsilon\omega}\}_\varepsilon$ and $\{V_{\theta^j+\varepsilon\omega}\}_\varepsilon$ are (respectively) increasing/decreasing sequences as $\varepsilon\searrow 0$. We can then consider their point-wise $\theta^j$-psh limits 
    $$
    \phi^{+}:=\lim_{\varepsilon\to 0}V_{\theta^j+\varepsilon\omega},\quad  \phi^{-}:=\left(\lim_{\varepsilon\to 0}V_{\theta^j-\varepsilon\omega}\right)^*.
    $$
    We also observe that, by Hartog's lemma, $\sup_X  \phi^{+}= \sup_X  \phi^{-}=0$.
   Now, by construction $\phi^+\geq V_{\theta^j}$. On the other hand, $V_{\theta^j}\geq \phi^+$ since $\phi^+$ is a candidate in the envelope. Hence $\phi^+=V_\theta$.\\
    By \cite[Corollary 3.4]{DNT21} we know that 
    \begin{eqnarray}\label{ma contact}
    \left(\theta^j- \varepsilon\omega+dd^cV_{\theta^j- \varepsilon\omega}\right)^n &=& \mathbf{1}_{\{V_{\theta^j- \varepsilon\omega}=0\}}\left(\theta^j-\varepsilon\omega\right)^n =  \mathbf{1}_{\{V_{\theta^j- \varepsilon\omega}=0\}}\left((\theta^j)^n + O(\varepsilon) \right)  
    \end{eqnarray}
   Since $V_{\theta^j-\varepsilon\omega}$ is increasing to $\phi^-$ the sets 
   $\{V_{\theta^j- \varepsilon \omega} = 0 \} $ increase as 
   $\varepsilon$ decreases to $0,$   outside of a pluripolar set. 
   Let 
$$W :=  \bigcup_{\varepsilon > 0 \ \mbox{small \ }} \{ V_{\theta^j - \varepsilon \omega} = 0 \}.$$ 
Then
    $\mathbf{1}_{\{V_{\theta^j - \varepsilon \omega}=0\}}$ is increasing to  
    $\mathbf{1}_{W} \leq \mathbf{1}_{\{\phi^{-}=0\}}$  outside of the same pluripolar set. Using the fact that $\left(\theta^j- \varepsilon\omega+dd^cV_{\theta^j- \varepsilon\omega}\right)^n$ converges weakly to $\left(\theta^j_{\phi^{-}}\right)^n$, $\phi^- \leq V_{\theta^j} \leq 0$ and that the form $(\theta^j)^n$ is non-negative on the set $ \{ V_{\theta^j} = 0 \}$,  from \eqref{ma contact} we deduce that 
    $$\left(\theta^j_{\phi^{-}}\right)^n = \mathbf{1}_{W}  (\theta^j)^n \leq \mathbf{1}_{\{\phi^- =0\}} (\theta^j)^n \leq  \mathbf{1}_{\{V_{\theta^j} =0\}} (\theta^j)^n = \left(\theta^j_{V_{\theta^j}}\right)^n. $$
    On the other hand by continuity of the volume function,

    $$
    \int_X (\theta^j_{\phi^{-}})^n=\lim_{\varepsilon\to 0} \int_X\left(\theta^j- \varepsilon\omega+dd^cV_{\theta^j- \varepsilon\omega}\right)^n=\lim_{\varepsilon\to 0} \vol(\theta^j-\varepsilon\omega)=\vol(\theta^j)=\int_X \left(\theta^j_{V_\theta}\right)^n.
    $$
    It then follows that $\left(\theta^j_{\phi^{-}}\right)^n =\left(\theta^j_{V_\theta}\right)^n$. By uniqueness of normalized solutions of Monge-Amp\`ere equations,  we infer that $\phi^-=V_{\theta^j}$. The claim is then proved\newline
    The proof now proceeds exactly as that in \cite[Theorem 2.6]{DDL6} replacing the forms $\theta^j$ by $\theta^j_k$. We give the details for the reader's convenience.\\ 
    Fix an open relatively compact subset $U$ of $\Omega$. By \eqref{amp}, we know that the functions $V_{\theta^j_{{k}}}$ are bounded on $U$. Fix $C>0$, $\varepsilon>0$ and consider
    $$
    f_j^{k,C,\varepsilon}:=\frac{\max(u_j^k-V_{\theta^j_{{k}}}+C,0)}{\max(u_j^k-V_{\theta^j_{{k}}}+C,0)+\varepsilon}, \quad j=1,\dots,n, \quad k\in \N^*
    $$
    and
    $$
    u_j^{k,C}:=\max(u_j^k,V_{{\theta^j_{{k}}}}-C).
    $$
    Observe that for $C,j$ fixed, the functions $u_j^{k,C}\geq V_{\theta^j_{{k}}}-C$ are uniformly bounded in $U$ (since $V_{\theta^j_{{k}}}$ are uniformly bounded in $U$) and converge in capacity to $u_j^C:= \max(u_j,V_{\theta^j}-C)$ as $k\to +\infty$ by Lemma \ref{lem:Capacity} below. 

Moreover $f_j^{k,C,\varepsilon}=0$ if $u_j^k\leq V_{\theta^j_{k}}-C$. By locality of the non-pluripolar product we can write
    $$
    f^{k,C,\varepsilon}\chi_k\theta^1_{{k},u_1^k}\wedge \dots\wedge \theta^n_{{k},u_n^k}=f^{k,C,\varepsilon}\chi_k \theta^1_{{k},u_1^{k,C}}\wedge \dots \wedge \theta^n_{{k},u_n^{k,C}},
    $$
    where $f^{k,C,\varepsilon}=f_1^{k,C,\varepsilon}\cdots f_n^{k,C,\varepsilon}$. For each $C,\varepsilon$ fixed the functions $f^{k,C,\varepsilon}$ are quasi-continuous, uniformly bounded (with values in $[0,1]$) and converge in capacity to $f^{C,\varepsilon}:=f_1^{C,\varepsilon}\cdots f_n^{C,\varepsilon}$ where $f_j^{C,\varepsilon}$ is defined by
    $$
    f_j^{C,\varepsilon}:=\frac{\max(u_j-V_{\theta^j}+C,0)}{\max(u_j-V_{\theta^j}+C,0)+\varepsilon}.
    $$
    With the information above we can apply \cite[Proposition 2.2]{DDL6} to get that
    $$
    f^{k,C,\varepsilon}\chi_k \theta^1_{{k},u_1^{k,C}}\wedge \dots \wedge \theta^n_{{k},u_n^{k,C}}\longrightarrow f^{C,\varepsilon}\chi \theta^1_{u_1^C}\wedge \dots \wedge \theta^n_{u_n^C}\quad \text{as } k\to +\infty,
    $$
    in the weak sense of measures on $U$. In particular since $0\leq f^{k,C,\varepsilon}\leq 1$ we have that
    \begin{align*}
        \liminf_{k\to+\infty}\int_X \chi_k \theta^1_{{k},u_1^k}\wedge \dots \wedge \theta^n_{{k},u_n^k}&\geq \liminf_{k\to +\infty} \int_U f^{k,C,\varepsilon}\chi_k \theta^1_{{k},u_1^{k,C}}\wedge \dots \wedge \theta^n_{{k},u_n^{k,C}}\\
        &\geq \int_U f^{C,\varepsilon} \chi \theta^1_{u_1^C}\wedge \dots \wedge \theta^n_{u_n^C}.
    \end{align*}
    Now, letting $\varepsilon\to 0$ and then $C\to +\infty$, we obtain
    $$
    \liminf_{k\to+\infty}\int_X \chi_k \theta^1_{{k},u_1^k}\wedge \dots \wedge \theta^n_{{k},u_n^k}\geq \int_U \chi \theta^1_{u_1}\wedge \dots \wedge \theta^n_{u_n}.
    $$
    Finally, letting $U$ increase to $\Omega$ and noting that the complement of $\Omega$ is pluripoar, we conclude the proof of the first statement of the theorem. Note that in the particular case $\chi_k=\chi\equiv 1$, we have
    $$ \liminf_{k\to+\infty}\int_X \theta^1_{{k},u_1^k}\wedge \dots \wedge \theta^n_{{k},u_n^k}\geq \int_X \theta^1_{u_1}\wedge \dots \wedge \theta^n_{u_n}.$$
    \newline
Thus we actually have equality in \eqref{eqn:Usc of MA masses} and the $\limsup$ is a $\lim$.\newline
    Now, let $B\in \R$ such that $\chi,\chi_k\leq B$. By \eqref{eqn:Lsc of MA operator} we get that
    $$
    \liminf_{k\to+\infty}\int_X (B-\chi_k)\theta^1_{{k},u_1^k}\wedge \dots\wedge \theta^n_{{k},u_n^k}\geq \int_X (B-\chi) \theta^1_{u_1}\wedge \dots\wedge \theta^n_{u_n}.
    $$
    Flipping the signs and using (equality in) \eqref{eqn:Usc of MA masses}, we conclude the following inequality, finishing the proof:
    $$
    \limsup_{k\to +\infty}\int_X \chi_k \theta^1_{{k},u_1^k}\wedge \dots\wedge \theta^n_{{k},u_n^k}\leq \int_X \chi \theta^1_{u_1}\wedge \dots\wedge \theta^n_{u_n}.
    $$
\end{proof}

\begin{lemma}\label{lem:Capacity}
    Let $u_k,v_k$ be two sequences of quasi-psh functions that converge in capacity respectively to $u,v$. Then
    $$
    \max(u_k,v_k)\longrightarrow\max(u,v)
    $$
    in capacity.
\end{lemma}
\begin{proof}
    Set $\varphi_k:=\max(u_k,v_k)$, $\varphi:=\max(u,v)$ and $\psi_k:=\sup_{j\geq k}\varphi_j$. Note that $\psi_k \searrow$ and 
    $\psi_k \geq \varphi_k$.
     We first observe that $\varphi_k\to \varphi$ in $L^1$. Indeed, as for any two real numbers $a,b$ we have $2\max(a,b)=a+b+\lvert a-b\rvert$, we get
    \begin{align*}
        2\lVert \varphi_k-\varphi \rVert_{L^1}&\leq \lVert u_k-u \rVert_{L^1} + \lVert v_k-v \rVert_{L^1} + \lVert \lvert u_k-v_k\rvert-\lvert u-v\rvert \rVert_{L^1}\\
        &\leq \lVert u_k-u \rVert_{L^1} + \lVert v_k-v \rVert_{L^1} + \lVert u_k-u +v_k-v \rVert_{L^1}\\
        &\leq 2\left(\lVert u_k-u \rVert_{L^1} + \lVert v_k-v \rVert_{L^1} \right).
    \end{align*}
    We deduce that $\varphi_k\to \varphi$ weakly as $u_k\to u,v_k\to v$ weakly.\newline
    Thus we infer that $\psi_k\searrow \varphi$, in particular $\psi_k \to \varphi$ in capacity. Let $\delta>0$. Clearly
    \begin{equation}\label{eqn:Cap_Conv1}
        \left\{\lvert \varphi_k-\varphi \rvert\geq \delta\right\}\subset     \left\{ \varphi_k\geq \delta+\varphi \right\}\cup \left\{ \varphi \geq \delta+ \varphi_k\right\}
    \end{equation}
    and
    \begin{equation}\label{eqn:Cap_Conv2}
        \left\{ \varphi_k\geq \delta+\varphi \right\}\subseteq \left\{ \psi_k\geq \delta+\varphi \right\}.
    \end{equation}
 Then we set $A:=\{u\geq v\}, B:=\{v\geq u\}$. On $A$ we have
    \begin{equation}\label{eqn:Cap_Conv3}
        \left\{ \varphi \geq \delta+ \varphi_k\right\}=\left\{ u \geq \delta+ \varphi_k\right\}\subseteq \left\{ u \geq \delta+ u_k\right\} = \left\{ \lvert u_k-u\rvert \geq \delta \right\}
    \end{equation}
    as $\varphi_k\geq u_k$. Similarly, on $B$ we get
    \begin{equation}\label{eqn:Cap_Conv4}
        \left\{ \varphi \geq \delta+ \varphi_k\right\}\subseteq \left\{ \lvert v_k-v\rvert \geq \delta \right\}.
    \end{equation}
    Combining \eqref{eqn:Cap_Conv3} and \eqref{eqn:Cap_Conv4} we get
    \begin{equation}\label{eqn:Cap_Conv5}
        \left\{ \varphi \geq \delta+ \varphi_k\right\}=\left(\left\{ \varphi \geq \delta+ \varphi_k\right\}\cap A \right)\cup\left(\left\{ \varphi \geq \delta+ \varphi_k\right\}\cap B \right)\subset \left\{ \lvert u_k-u\rvert \geq \delta \right\} \cup \left\{ \lvert v_k-v\rvert \geq \delta \right\}.
    \end{equation}
    Thus \eqref{eqn:Cap_Conv1}, \eqref{eqn:Cap_Conv2} and \eqref{eqn:Cap_Conv5} leads to
    $$
    \left\{\lvert \varphi_k-\varphi \rvert\geq \delta\right\}\subseteq \left\{ \psi_k-\varphi\geq \delta \right\} \cup \left\{ \lvert u_k-u\rvert \geq \delta \right\} \cup \left\{ \lvert v_k-v\rvert \geq \delta \right\}.
    $$
    As $u_k\to u, v_k\to v, \psi_k\to \varphi$ in capacity, we conclude that $\varphi_k\to \varphi$ in capacity thanks to the subadditivity property of the Monge-Ampère capacity.
\end{proof}

\medskip
 \subsection{Envelopes and model potentials} \label{subsec env}
If $f$ is a function on $X$, we define the  envelope of $f$ in the class ${\rm PSH}(X,\theta)$ by
\[
P_{\theta}(f) := \left(\sup \{u\in {\rm PSH}(X,\theta) \; : \;  u\leq f\}\right)^*,
\]
with the convention that $\sup\emptyset =-\infty$. Observe that $P_{\theta}(f)\in {\rm PSH}(X,\theta)$ if and only if there exists some $u\in {\rm PSH}(X,\theta)$ lying below $f$. Note also that $V_{\theta}=P_{\theta}(0)$, and that $P_{\theta}(f+C)= P_{\theta}(f)+C$ for any constant $C$. 

\noindent In the particular case $f=\min(\psi,\phi)$, we denote the envelope as
$ P_\theta(\psi,\phi):=P_\theta(\min(\psi,\phi))$.
We observe that $ P_\theta(\psi,\phi)=  P_\theta(P_\theta(\psi),P_\theta(\phi))$, so w.l.o.g. we can assume  $\psi,\phi$ are two $\theta$-psh functions.

\medskip
Starting from the  ``rooftop envelope'' $ P_\theta(\psi,\phi)$ we introduce
$$P_\theta[\psi](\phi) := \Big(\lim_{C \to \infty}P_\theta(\psi+C,\phi)\Big)^*.$$
It is easy to see that $P_\theta[\psi](\phi)$ only depends on the singularity type of $\psi$. When $\phi = V_\theta$, we will simply write $P_\theta[\psi]:=P_\theta[\psi](V_\theta)$, and we refer to this potential as the \emph{envelope of the singularity type} $[\psi]$.

Since $\psi - \sup_X \psi \leq P_\theta[\psi]$, we have that $[\psi] \leq [P_\theta[\psi]]$ and typically equality does not happen. When $[\psi] = [P_\theta[\psi]]$, we say that $\psi$ has \emph{model singularity type}. In the (more particular) case $\psi = P_\theta[\psi]$ we say that $\psi$ is a \emph{model potential}.\\
It is worth to mention that given any $\theta$-psh function $\psi$ with positive mass, the associated envelope $P_\theta[\psi]$ is in fact a model potential \cite[Theorem 3.14]{DDL6}. 
\medskip

From now on, (otherwise stated), $\phi$ will denote a model potential with strictly positive mass, i.e. $\int_X \theta_\phi^n>0$. We say that a $\theta$-psh function $\varphi$ has $\phi$-relative minimal singularities if $\varphi \simeq \phi$.

\begin{remark}
Let $\psi \in \PSH(X,\theta)$  with  $\sup_X \psi = 0,$ for  $ N \in \bf{N}$ 
we set $P_N := P_\theta(\psi + N, V_\theta).$  On one hand we have $P_N  \sim \psi,$ on the other hand $P_N$ is increasing to $P[\psi],$ then by Remark \ref{rk:conv_incr} $\theta_{P_N}^n$ converges weakly to $\theta_{P[\psi]}^n,$ hence $\int_X \theta_\psi^n = \int_X \theta_{P[\psi]}^n.$  \label{model} \end{remark} 

\begin{definition}
Given a model potential $\phi$, the relative full mass class $\mathcal{E}(X,\theta,\phi)$ is the set of all $\theta$-psh functions $u$ such that $u$ is more singular than $\phi$ and $\int_X \theta_u^n=\int_X \theta_{\phi}^n$.  We will denote simply by $\mathcal{E}(X,\theta)$ the space $\mathcal{E}(X,\theta,V_\theta).$
\end{definition}

As pointed out in \cite{GZ07}, for potential theoretic reasons, it is natural to consider weighted subspaces of $\mathcal E(X,\theta,\phi)$.

 A weight is a continuous strictly increasing function  $\chi: [0,+\infty) \rightarrow [0,+\infty)$ such that $\chi(0)=0$ and $\chi(+\infty)=+\infty$. Denote by $\chi^{-1}$ its inverse function, i.e. such that $\chi(\chi^{-1}(t))= t$ for all $t\geq 0$.

\medskip

We fix $\phi$ a model potential and we let $\mathcal{E}_{\chi}(X,\theta,\phi)$ denote the set of all $u\in \mathcal{E}(X,\theta,\phi)$ such that 
\[
E_{\chi}(u,\phi):=\int_X \chi(|u-\phi|) \theta_u^n <\infty. 
\]
When $\phi=V_{\theta}$, we denote $\mathcal{E}(X,\theta)=\mathcal{E}(X,\theta,V_{\theta})$, $\mathcal{E}_{\chi}(X,\theta)=\mathcal{E}_{\chi}(X,\theta,V_{\theta}$) and $E_{\chi}(u)=E_{\chi}(u,V_{\theta})$. Compared to \cite{GZ07},  we have changed the sign of the weight, but the weighted classes are the same. \\Also, in the special case $\chi(t)=t^p$, $p>0$, we simply denote the relative energy class with $\mathcal{E}^p(X,\theta,\phi)$
and the corresponding relative energy $E_p(u,\phi)$.
\medskip

\begin{remark}\label{rk:conv_incr}
Under the assumptions of Theorem \ref{thm:weak convergence} we further assume that for all $j \in  \{ 1, \ldots,n \} $ and for $k$ large enough, $u^k_j$ is more singular than $u_j$, then $$\theta^1_{u_1^k} \wedge \theta^2_{u_2^k} \wedge \ldots  \wedge \theta^n_{u_n^k} \to \theta^1_{u_1} \wedge \theta^2_{u_2} \wedge \ldots  \wedge \theta^n_{u_n}$$ in the weak sense of measures. Indeed, by \cite[Theorem 1.2]{WN19}, \cite[Theorem 3.3]{DDL6}, if $u^k_j$ is more singular than $u_j$ we have 
$$\int_X \theta^1_{u_1^k} \wedge \theta^2_{u_2^k} \wedge \ldots  \wedge \theta^n_{u_n^k} \leq  \int_X  \theta^1_{u_1} \wedge\theta^2_{u_2} \wedge \ldots  \wedge \theta^n_{u_n}.$$
This means that the second statement of Theorem \ref{thm:weak convergence} above holds with $\chi_k=\chi\equiv 1$.\\
The same conclusion holds if $u^k_j$, $u_j\in \mathcal{E}(X, \theta, \phi_j)$ (where $\phi_j$ are model potentials) since by \cite[Proposition 3.1]{DDL3} we know that
$$ \int_X \theta^1_{u_1^k} \wedge \theta^2_{u_2^k} \wedge \ldots  \wedge \theta^n_{u_n^k} = \int_X  \theta^1_{u_1} \wedge\theta^2_{u_2} \wedge \ldots  \wedge \theta^n_{u_n}.$$
\end{remark}
\subsection{Plurisubharmonic geodesics}
We next recall the definition/construction in \cite{DDL1} of plurisubharmonic geodesics.

For a curve $(0,1) \ni t \mapsto u_t \in \PSH(X,\theta)$  we define 
\begin{equation}\label{eq: complexified curve}
X\times A \ni (x,z) \mapsto U(x,z) := u_{\log |z|}(x),	
\end{equation}
where  $A= \{z\in \C, \; 1 < |z| < e \}$ and $\pi: X\times A \rightarrow X$ is the projection on the first factor. 
\begin{definition}
	We say that $t\mapsto u_t$ is a subgeodesic if $(x,z) \mapsto U(x,z)$ is a $\pi^{*}\theta$-psh function on $X\times A$.
\end{definition}

\begin{definition}
	For  $\varphi_0,\varphi_1 \in \PSH(X,\theta)$, we let $\mathcal{S}_{[0,1]}(\varphi_0,\varphi_1)$ denote the set of  all subgeodesics $(0,1) \ni t \mapsto u_t$ such that $\limsup_{t\to 0^+} u_t\leq \varphi_0$ and $\limsup_{t\to 1^-} u_t\leq \varphi_1$. 
\end{definition}

Let $\varphi_0,\varphi_1 \in \PSH(X,\theta)$. For $(x,z)\in X\times A$ we define
$$
\Phi(x,z) :=  \sup \{ U(x,z) \; :\;  U \in \mathcal{S}_{[0,1]}(\varphi_0,\varphi_1) \}. 
$$
The curve $t\mapsto \varphi_t$ constructed from $\Phi$ via \eqref{eq: complexified curve} is called the plurisubharmonic (psh) geodesic segment connecting $\varphi_0$ and $\varphi_1$.  

Geodesic segments connecting two general $\theta$-psh functions may not exist. If $\varphi_0, \varphi_1 \in \mathcal{E}^p(X,\theta)$, it was shown in \cite[Theorem 2.13]{DDL1} that $P(\varphi_0,\varphi_1) \in \mathcal{E}^p(X,\theta)$. Since $P(\varphi_0,\varphi_1) \leq \varphi_t$, we obtain that $t \to \varphi_t$ is a curve in $\mathcal{E}^p(X,\theta)$. Also, each subgeodesic segment is convex in $t$:  
\[
\varphi_t\leq \left (1-t\right )\varphi_0 + t\varphi_1, \ \forall t\in [0,1]. 
\]
Consequently the upper semicontinuous regularization (with respect to both variables $x,z$) of $\Phi$ is again in $\mathcal{S}_{[0,1]}(\varphi_0,\varphi_1)$, hence so is $\Phi$. If $\varphi_0,\varphi_1$ have the same singularities then the geodesic $\varphi_t$ is Lipschitz on $[0,1]$ (see \cite[Lemma 3.1]{DDL1}): 
\begin{equation}
	\label{eq: Lip}
	|\varphi_t-\varphi_s| \leq |t-s| \sup_{X} |\varphi_0-\varphi_1|, \ \forall t,s \in [0,1]. 
\end{equation}

 \section{Entropy}\label{sec: entr}
 We recall that given two positive probability measures $\mu, \nu$, the relative entropy $\Ent(\mu, \nu)$ is defined as  
\[
\Ent(\mu, \nu) := \int_X \log\left (\frac{d\mu}{d\nu}\right) d\mu,
\]
if $\mu$ is absolutely continuous with respect to $\nu$, and $+\infty$ otherwise. \\
\begin{remark}
Let $\mu, \nu$ positive probability measure with $\mu := f \nu $ absolutely continuous with respect to $\nu.$ Then $\Ent(\mu,\nu) < + \infty$ if and only if $f\log f \in L^1(X,\nu),$ in fact if $f < 1,$ $f \log f$ is bounded, and if $f \geq 1,$ $f \log f \geq 0.$ \label{L^1}       
\end{remark}

Once for all, we normalize the K\"ahler form $\omega$ such that $\int_X \omega^n=1$. We consider $u\in \psh(X, \theta)$ such that $\theta_u^n=f \omega^n$, $0 \leq f$ and $m_u:=\int_X \theta_u^n >0$. Then $ u\in \mathcal{E}(X, \theta, \phi)$, where $\phi$ is the model potential with $\int_X \theta_\phi^n=\int_X \theta_u^n $, and $m_u^{-1}\theta_u^n$ s a probability measure. We then define the $\theta$\emph{-entropy} of $u$ as
\begin{equation}\label{eq: entropy}
 {\rm Ent}_\theta(u):= {\rm Ent}( m_u^{-1}\theta_u^n, {\omega^n})=\int_X \log \left( \frac{m_u^{-1}\theta_u^n}{\omega^n}  \right) m_u^{-1}\theta_u^n= m_u^{-1}  \int_X f\log f  \omega^n-\log m_u. 
 \end{equation}
By Jensen inequality we have ${\rm Ent}_\theta(u)\geq 0$. Also, observe that the definition of the $\theta$-entropy does depend on the chosen volume form $\omega^n$ but its finiteness does not. \\
Also, the expression in \eqref{eq: entropy} coincides with the definition of entropy in \cite{DGL20} when $P[u]=V_\theta$, i.e. when $u\in \mathcal{E}(X, \theta)$. The definition in \eqref{eq: entropy} is indeed a generalisation that allows to consider any $\theta$-psh function not necessarily of full mass.\\
\medskip
More generally, given two $\theta$-psh functions $u,v$ with $m_u, m_v>0$ we define
$${\rm Ent}_\theta(u,v) :={\rm Ent}( m_u^{-1}\theta_u^n, m_v^{-1}\theta_v^n). $$
Also, if no confusion can arise, we simply write ${\rm Ent} (u) $ and ${\rm Ent} (u,v) $.

\begin{definition}\label{def: entropy}
    We say that $u\in \psh(X, \theta)$ with $m_u>0$ has \emph{finite $\theta$-entropy} if ${\rm Ent}_\theta(u)<+\infty$. We denote by ${\rm Ent}(X, \theta)$ the set of all $\theta$-psh functions having finite $\theta$-entropy.
\end{definition}
 \noindent By \eqref{eq: entropy}, ${\rm Ent}_\theta(u)<+\infty$ if and only if $\int_X f\log f\omega^n<+\infty$ or equivalently $\int_X (f+1)\log (f+1)\omega^n<+\infty$.
\smallskip

\noindent We recall the following result from \cite{DTT23} which will  be useful in the following:

 \begin{lemma}
    \label{lem:Holo}
     Let $\pi:Y\to X$ a bimeromorphic and holomorphic map and assume that $\tilde{\omega}$ is K\"ahler form on $Y$ normalized with volume equal to $1$. Then
     \begin{itemize}
     \item[(i)] $\Ent (\mu, \nu)=\Ent (\pi^* \mu, \pi^* \nu) $, for any two non-pluripolar probability measures $\mu, \nu$.
     \item[(ii)] If $\Ent_\theta(\varphi) <+\infty$, then $\Ent(m_\varphi^{-1}  \pi^* \theta^n_\varphi , \tilde{\omega}^n)<+\infty$. In particular $\Ent_{\pi^* \theta} (\pi^* \varphi)<+\infty$.
     \end{itemize}
 \end{lemma}
By $\pi^*\mu$ we mean the pushforward by $\pi^{-1}$ of $\mathbf{1}_{X\setminus Z}\mu$ where $Z$ is the indeterminacy locus of $\pi^{-1}$ (see also \cite[lines after Definition 1.3]{BBEGZ19}).
 We refer to \cite[Lemma 2.11]{DTT23} for a proof.

\section{Convexity of the Mabuchi functional in big classes}\label{sec: convex}

Assume $\theta$ to be a smooth form such that $\{\theta\}$ is big. W.l.o.g. we normalize the K\"ahler form $\omega$ such that $\int_X\omega^n=1$.
Following \cite[Section 2]{BBJ15}, given a $\theta$-psh function $\varphi$, we say that $\theta_\varphi$ has well-defined Ricci curvature if its Monge-Ampère measure $\theta_\varphi^n$ corresponds to a singular metric on $K_X$, i.e. if locally
$$
\theta_\varphi^n=e^{-f}i^{n^2}\Omega\wedge \bar{\Omega}
$$
with $f\in L^1_{loc}$ and where $\Omega$ is any nowhere zero local holomorphic section of $K_X$. The Ricci curvature is then locally given by $$\Ric(\theta_\varphi):=\Ric (\theta_\varphi^n)= dd^c f.$$
The local currents  $dd^c f$ glue together and define a (global) closed real  $(1,1)$-current $\Ric(\theta_\varphi)$, which recovers the usual definition of the Ricci curvature when $\theta_\varphi$ is K\"ahler. With this choice of norma\-lization, if $\theta$ and $\theta_\varphi$ have well-defined Ricci curvature we have
$$\Ric(\theta_\varphi) = \Ric(\theta) -dd^c \log \left(\frac{\theta_\varphi^n}{\theta^n}\right).$$

Moreover for every current $\theta_\varphi$ with well defined Ricci curvature the cohomology class of $\Ric(\theta_\varphi)$ is always the first Chern class. 
We say that $\theta_\varphi$ has \emph{good Ricci curvature} if it has well-defined Ricci curvature and there exists real closed positive $(1,1)$-currents $T_1, T_2$ such that  
$$\Ric(\theta_\varphi)= T_1-T_2.$$
Note that if $\theta_\varphi$ has good Ricci curvature then $\int_X\theta_\varphi^n>0$ and $\theta_\varphi^n=e^g \omega^n$ where $g$ is difference of quasi-plurisubharmonic functions. In particular, \cite[Theorem 1.1]{GZh15} implies that $\theta_\varphi^n$ has $L^p$ density for $p>1$. Hence \cite[Theorem 1.4]{DDL2} gives that $\varphi$ has $\phi:=P_\theta[\varphi]$-relative minimal singularities, where $\phi$ is the model type envelope such that $\varphi\in \cE(X,\theta,\phi)$.

\smallskip
Let now $\varphi$ be a $\theta$-psh function with good Ricci curvature and let $\phi=P_\theta[\varphi]$ be the model potential associated to $\varphi$. 
\medskip

For $u,v\in \PSH(X,\theta)$ with $|u-v|$ bounded,  we then consider 
\begin{equation}\label{eq: energy gen down}
 E(\theta; u,v) := \frac{1}{(n+1)m_\phi} \sum_{k=0}^n \int_{X} (u-v)\, \theta_u^{k} \wedge \theta_{v}^{n-k}
 \end{equation}
 and
 \begin{equation}\label{eq: energy ricci down}
  E_{\Ric(\theta_\varphi)}(\theta; u,v) :=  \frac{1}{n\, m_\phi} \sum_{k=0}^{n-1} \int_{X} (u-v)\, \theta_u^{k} \wedge \theta_v^{n-k-1}\wedge \Ric(\theta_\varphi),
 \end{equation}
 where each integral in the left-hand side is defined as 
 \begin{equation}\label{def: energyricci}
  \int_{X} (u-v)\, \theta_u^{k} \wedge \theta_v^{n-k-1}\wedge T_1 - \int_{X} (u-v)\, \theta_u^{k} \wedge \theta_v^{n-k-1}\wedge T_2.  
   \end{equation}
Similarly, given $\alpha=T_1-T_2$, with $T_i$ closed positive $(1,1)$-currents, we set
    $$
    E_{\alpha}(\theta; u,v):=\frac{1}{n\, m_\phi} \sum_{k=0}^{n-1}  \Big(\int_{X} (u-v)\, \theta_u^{k} \wedge \theta_v^{n-k-1}\wedge T_1- \int_{X} (u-v)\, \theta_u^{k} \wedge \theta_v^{n-k-1}\wedge T_2\Big).
    $$
To avoid any ambiguity, we recall that the wedge products between positive currents has to be understood as the non-pluripolar product. We observe that the above definition is independent of the choice of the two positive currents $T_1, T_2$. 
Note also that thanks to \cite[Lemma 5.6]{DDL6} if $\phi$ is a model potential and $u \in \mathcal{E}(X,\theta,\phi)$ then 
$E(\theta,u,\phi)$ is finite if and only if $u \in \mathcal{E}^1(X,\theta,\phi).$
 
 \medskip

Taking inspiration from the Kähler setting (see also \cite{DNL21}), we define the Mabuchi functional relative to a $(X,\theta_\varphi)$ as
\begin{equation}
    \label{eq: K-energy}
 	\cM_{\theta, \varphi}(u) := \bar{S}_{\varphi} E(\theta; u,{\varphi}) - n E_{\Ric (\theta_\varphi)} (\theta; u,{\varphi}) +\Ent(u,\varphi), \qquad u \simeq \varphi
\end{equation}
for any $u\in \mathcal{E}(X,\theta,\phi)$ with $\phi$-relative minimal singularities where $\Ent(u,\varphi):=\Ent(\theta_u^n m_\phi^{-1}, \theta_\varphi^n m_\phi^{-1})$ and
 \[
\bar{S}_{\varphi}:= \frac{n}{m_\phi}\int_{X}\Ric(\theta_\varphi) \wedge \theta_\varphi^{n-1}.
 \]

\medskip
We start with a Proposition which generalizes \cite[Theorem 3.12]{DDL1}:

\begin{prop}\label{prop: energy linear}
Let $u_0, u_1\in \cE^1(X, \theta)$ and let $u_t$ be the psh geodesic joining $u_0, u_1$.
Then $t\rightarrow E(\theta; u_t, V_\theta)$ is linear.
\end{prop}

\begin{proof}
    For $i=0,1$ and $C>0$, we set $u_i^C:= \max(u_i, V_\theta-C )$. Let $u_{t}^C$ be the psh geodesic joining $u_0^C$ and $u_1^C$. Observe that $u_t^C$ has minimal singularities. \\
    We claim that $u_t^C$ decreases to $u_t$. Indeed, since $u_i^C\geq u_i$, by comparison principle, $u_t^C$ is a decreasing sequence such that $u_t^C \geq u_t$. Hence $u_t^C \searrow w_t$ for some $\theta$-psh function $w_t\geq u_t$ joining $u_0$ and $u_1$. By maximality of the geodesic we infer that $w_t=u_t$. By \cite[Theorem 3.12]{DDL1}, $t\rightarrow E(\theta; u_t^C, V_\theta)$ is linear. Moreover by \cite[Lemma 5.7]{DDL6} we know that $E(\theta; u_t^C, V_\theta)$ converges to $E(\theta; u_t, V_\theta)$. It then follows that $t\rightarrow E(\theta; u_t, V_\theta)$, as limit of a linear function, is linear as well.
\end{proof}

The next statement slightly generalizes the cocycle property in \cite[Theorem 5.3]{DDL6}:
\begin{prop}\label{prop: cocy model}
Let $\phi$ be a model potential. Then for any $u,v\in \mathcal{E}^1(X,\theta, \phi)$, we have
\begin{equation}\label{prop:cocycleE1}
{E}(\theta; u, \phi)-{E}(\theta; v, \phi)={E}(\theta; u, v).
\end{equation}
\end{prop}
For the proof we adapt the arguments in \cite[Proposition 2.5]{DDL1} We will refer to \eqref{prop:cocycleE1} as the \emph{cocycle property}.

\begin{proof}
	By \cite[Corollary 3.16]{DDL6} for $ k \in \{0,\ldots,n\},$ 
 $$\int_X \theta_u^k \wedge \theta_{\phi}^{n-k} = \int_X \theta_v^k \wedge \theta_{\phi}^{n-k}  $$ so we can assume that  $\max(u,v) \leq \phi \leq 0.$
 For $C>0$, we set $u^C:= \max(u, \phi -C)$, $v^C:= \max(v, \phi -C)$. By locality we have \{0,\ldots,n\}$ $
$$\id_{\{u>\phi-C\}} \theta_{u^C}^n= \id_{\{u>\phi-C\}} \theta_{u}^n,$$
and more generally, for any $k\in \{0,...,n\}$
\begin{equation}\label{loc property}
\id_{\{\min(u,v)>\phi-C\}} \theta_{u^C}^k \wedge \theta_{v^C}^{n-k} = \id_{\{\min(u,v)>\phi-C\}} \theta_{u}^k \wedge \theta_{v}^{n-k}.
\end{equation}
Since $\int_X \theta_{u}^n=\int_X \theta_{u^C}^n$ we can write
\begin{eqnarray}
\label{eq: E1 by mass}
\lim_{C\to +\infty} C\int_{\{u\leq \phi-C\}}\theta_{u^C}^n & = & \lim_{C\to +\infty} C\int_{\{u\leq \phi-C\}}\theta_{u}^n  \\
&\leq & \lim_{C\to +\infty} \int_{\{u\leq\phi-C\}} (\phi-u)\theta_{u}^n = 
\int_{\{u = - \infty\}} (\phi-u)\theta_{u}^n = 0.
\nonumber
\end{eqnarray}

 By \cite[Theorem 5.3]{DDL6} we have that
 $${E}(\theta; u^C, \phi)-{E}(\theta; v^C, \phi)={E}(\theta; u^C, v^C).$$
 
 By \cite[Lemma 4.12]{DDL2} we already know that ${E}(\theta; u^C, \phi)$ and ${E}(\theta; v^C, \phi)$ decrease to ${E}(\theta; u, \phi)$ and ${E}(\theta; v, \phi)$, respectively. 
 We want to prove that ${E}(\theta; u^C, v^C)$ decreases to ${E}(\theta; u, v)$, i.e. that for any $k\in \{0,...,n\}$, 
    \begin{equation}\label{eq: basic I energy 11}
		\lim_{C\to +\infty}\int_X (u^C-v^C) \theta_{u^C}^k \wedge \theta_{v^C}^{n-k}  =\int_X (u-v) \theta_{u}^k \wedge \theta_{v}^{n-k}. 
	\end{equation}
    Clearly, it suffices to check that
	\begin{equation}\label{eq: basic I energy 1}
		\lim_{C\to +\infty}\int_X (u^C-\phi) \theta_{u^C}^k \wedge \theta_{v^C}^{n-k}  =\int_X (u-\phi) \theta_{u}^k \wedge \theta_{v}^{n-k} 
	\end{equation}

 and that 
 \begin{equation}\label{eq: basic I energy 1a}
		\lim_{C\to +\infty}\int_X (\phi-v^C) \theta_{u^C}^k \wedge \theta_{v^C}^{n-k}  =\int_X (\phi-v) \theta_{u}^k \wedge \theta_{v}^{n-k}.
	\end{equation}
In the following we prove \eqref{eq: basic I energy 1}. The same arguments will give \eqref{eq: basic I energy 1a}.\\
	We decompose the integral into two parts $\int_{\{\min(u,v)>\phi-C\}}$ and $\int_{\{\min(u,v)\leq \phi-C\}}$, by  the locality property \eqref{loc property}
 we have  $$\int_{\{\min(u,v)>\phi-C\}} (u^C-\phi) \theta_{u^C}^k \wedge \theta_{v^C}^{n-k} $$
  $$ = \int_{\{\min(u,v)>\phi-C\}} (u-\phi) \theta_{u}^k \wedge \theta_{v}^{n-k} \to$$
   $$\int_X (u-\phi) \theta_{u}^k \wedge \theta_{v}^{n-k} $$ as $C \to + \infty.$
Noting that $\{\min(u,v)\leq \phi-C\} \subseteq \{u\leq \phi-C\}\cup \{v\leq \phi-C\}$, we see that proving \eqref{eq: basic I energy 1} boils down to showing that 
	\begin{equation}
		\label{eq: energy estimate 1}
		\lim_{C\to +\infty} C \int_{\{u\leq \phi-C\}}\theta_{u^C}^k \wedge \theta_{v^C}^{n-k} =0,\  \text{and}\ \lim_{C\to +\infty} C \int_{\{v\leq \phi-C\}}\theta_{u^C}^k \wedge \theta_{v^C}^{n-k} =0, \  \forall k. 
	\end{equation}

	We will prove the first equality and the same arguments apply to prove the second one. 
	Observing that $\phi-C \leq v^C \leq \phi$ we have the inclusion 
	\[
	\{u\leq \phi -C\} \subset \left \{u^C\leq \frac{v^C+\phi -C}{2}\right \} \subset  \{u\leq \phi-C/2\}.
	\]
	Using the partial comparison principle \cite[Proposition 3.22]{DDL6} and that 
    $$  2^{k-n}  \theta_{v^C}^{n-k} \leq \left(\frac{\theta}{2}+dd^c \frac{v^C}{2}\right)^{n-k} \leq   \theta_{\frac{v^C+\phi -C}{2}}^{n-k}$$ 
    we get 
	\begin{eqnarray*}
		C \int_{\{u\leq \phi-C\}}\theta_{u^C}^k \wedge \theta_{v^C}^{n-k}  &\leq & 2^{n-k} C   \int_{\left\{u^C\leq \frac{v^C+\phi-C}{2}\right\}}\theta_{u^C}^k \wedge \theta_{\frac{v^C+\phi -C}{2}}^{n-k}\\
		&\leq & 2^{n-k} C  \int_{\left\{u^C\leq \frac{v^C+V_{\theta}-C}{2}\right\}}\theta_{u^C}^n\\
		&\leq & 2^{n-k} C  \int_{\{u\leq \phi-C/2\}}\theta_{u^C}^n\\
	\end{eqnarray*}
From the above and \eqref{eq: E1 by mass} we obtain \eqref{eq: energy estimate 1}, completing the proof. 
\end{proof}

\subsection{From model to divisorial singularities} 
\label{sec:Analytic_Singularities}

We introduce a set of model potentials with which we will work through the paper.
\begin{definition}\label{defi:N_theta}
    Let $\mN_\theta\subset \PSH(X,\theta)$ be the set of all model potential $\phi$ such that there exists a modification (i.e. bimeromorphic holomorphic map) $\pi:Y\to X$  from $Y$ a compact Kähler manifold of dimension $n$ such that
    $$
    \pi^*\theta_\phi= [F]+S
    $$
    for an effective $\R$-divisor $F$ and a closed, positive current $S$ with minimal singularities, representing a big and nef class.
\end{definition}

\begin{remark}
In the definition above we restrict our attention to model potentials since if a $\theta$-psh function $u$ admits a modification $\pi:Y\to X$ such that $ \pi^*\theta_u= [F]+S$, for an effective $\R$-divisor $F$ and a closed, positive current $S$ with minimal singularities, representing a big and nef class, then $u$ is of model type. \\
Indeed, since $\pi^*\theta_u= [F]+S$ and $\{S\}$ is big, we have $\int_X \theta_u^n= \int_Y S^n>0$. By \cite[Theorem 1.3]{DDL2} $u\in \cE(X, \theta, P[u])$. Moreover, by \cite[Lemma 5.1]{DDL6} we infer that $u$ and $P[u]$ have the same multiplier ideal sheaf and in particular $u\circ \pi$ and $P[u]\circ \pi$ have the same Lelong numbers \cite[Theorem A]{BFJ08}. Thus $\pi^*\theta_{P[u]}-[F]$ is a positive and closed $(1,1)$-current in the cohomology class $\{S\}$.
Thus $\pi^*\theta_{P[u]}= [F]+\tilde{S}$, where $\{\tilde{S}\}=\{S\}$. Since $P[u]$ is less singular than $u$, $\tilde{S}$ is less singular than $S$, hence it has minimal singularities.
\end{remark}

The following lemma lists some properties of $\mN_\theta$.

\begin{lemma}\label{lem:properties N_theta} The followings hold:
    \begin{itemize}
        \item[i)] For $\phi\in \mN_\theta$ we have
        $$
        m_\phi= \vol(S) 
        $$
        where $ \vol(S) $ is the volume of the nef and big class $\{S\}$. In particular $m_\phi>0$, and letting $\eta$ be a smooth and closed form representing $\{S\}$ we have $m_\phi=\int_Y \eta^n$.
        \item[ii)] if $\psi\in\PSH(X,\theta)$ is a function with analytic singularities such that $\int_X \theta_\psi^n>0$, then $\phi:=P_\theta[\psi]\in\mN_\theta$ and any associated big and nef class admits bounded potentials.
    \end{itemize}
\end{lemma}

\begin{proof}
   We start proving $(i)$. As the non-pluripolar product does not charge pluripolar sets such as divisors we have
    $$
    m_\phi=\int_X\theta_\phi^n=\int_Y S^n=\vol(S)
    $$
    where the last equality follows from the fact that $S$ has minimal singularities. As the class $\{S\}$ is big we deduce that $m_\phi>0$ and the equality $\vol(S)=\int_Y\eta^n$ for a smooth and closed form $\eta$ representing $\{S\}$ follows from the fact that the class is nef.\newline
    We now prove $(ii)$. The local holomorphic functions $f_1,\dots,f_N$ such that
    $$
    \psi= c\log \sum_{j=1}^N \lvert f_j\rvert^2 +g
    $$
    as in \eqref{eqn:An_Sing} generates a coherent ideal sheaf $\mathcal{I}$. Taking a log resolution of $(X,\mathcal{I})$ yields a modification $\pi: Y\to X$ such that
    $$
    \pi^*\theta_\psi= T+c[F]
    $$
    where $F$ is an effective divisor such that $\mathcal{O}_Y(-F)=\pi^{-1}\mathcal{I}$ and $T=\eta+dd^c \tilde{\psi}$ for a smooth and closed form $\eta$ and for $\tilde{\psi}\in \PSH(Y,\eta)\cap L^\infty(Y)$. \\
    Indeed,
$$
\psi\circ \pi=c\log \sum_{j=1}^N \lvert f_j\circ \pi\rvert^2 + g\circ \pi= c\log \left(\prod_{l=1}^M \lvert u_l\rvert^2\sum_{j=1}^N \lvert v_j \rvert^2 \right)+g\circ \pi=c\sum_{l=1}^M\log \lvert u_l\rvert^2 + c\log \sum_{j=1}^N \vert v_j\rvert^2 + g\circ \pi
$$
where $u_l$, $l=1, \cdots,M$ are the holomorphic functions which divides all the $f_j\circ \pi$ and which define the divisor $F$. Note that $v_j$, $j=1, \cdots, N$, do not any common zeros by construction, and so $\log \sum_{j} \vert v_j\rvert^2$ is bounded.

    Next, by \cite[Proposition 5.24]{DDL6}, we can infer that $\psi$ has model type singularities, i.e. $|P_\theta[\psi]-\psi|\leq C$, $C>0$. Set $\phi:=P_\theta[\psi]$. By \cite[Lemma 5.1]{DDL6} we infer that $\psi$ and $\phi$ have the same multiplier ideal sheaf and in particular $\psi\circ \pi$ and $\phi\circ \pi$ have the same Lelong numbers \cite[Theorem A]{BFJ08}. Thus $\pi^*\theta_\phi-c[F]$ is a positive and closed $(1,1)$-current in the cohomology class $\{T\}$. Since $\psi\circ \pi$ and $\phi\circ \pi$ have the same singularities, we have
    $$
    \pi^*\theta_\phi=S+c[F]
    $$
    for a closed and positive current $S$ with the same singularities of $T$, i.e. $S=\eta+dd^c \tilde{\phi}$ for  $\tilde{\phi}\in \PSH(Y,\eta)\cap L^\infty(Y)$. In particular $S$ is a current with minimal singularities, hence by combining \cite[Propositions 3.2 and 3.6]{Bou04} we find that the class $\eta$ is nef. Finally we have
    $$
    0<\int_X \theta_\psi^n=\int_X \theta_\phi^n=\int_Y S^n=\vol(S)
    $$
    by the same calculation performed in the proof of (i). Hence $\{S\}$ is big, which concludes the proof.
\end{proof}

Now, suppose $\phi\in \mN_\theta$, and consider $\pi:Y\to X$ a modification such that
\begin{equation}
    \label{eqn:Almost_Anal_Sing}
    \pi^*\theta_\phi=\eta_{\tilde{\phi}} +[F]
\end{equation}
for an effective $\R$-divisor $F$, for $\eta$ closed smooth $(1,1)$-form such that $\{\eta\}$ is big and nef, and for $\tilde{\phi}\in \PSH(Y,\eta)$ normalized such that $\sup_Y\tilde{\phi}=0$. Let also $\tilde{\omega}$ be a fixed K\"ahler form on $Y$ normalized such that $\int_Y \tilde{\omega}=1$. 
\medskip 

Let $E_1,\dots,E_m$ be the exceptional divisors of $\pi:Y\to X$, $a_j>0$ such that $K_{Y/X}=\sum_{j=1}^m a_j E_j.$
Recall  that at the level of cohomology classes we have 
\begin{equation}\label{relcan} 
K_Y = \pi^*K_X + K_{Y/X}. \end{equation}

Let $h_j$ be smooth metrics on $\mathcal{O}_Y(E_j)$ such that the curvature form $\Theta:=\sum_{j=1}^m a_j\, \Theta(h_j)$ satisfies
$$
\Theta=\pi^*\Ric(\omega)-\Ric(\tilde{\omega}).
$$
Then there exist holomorphic sections $s_j$ of $\mathcal{O}_Y(E_j)$ such that
\begin{equation}\label{eqn:pullback of Ricci}
    \Ric(\pi^*\omega)
    =\pi^*\Ric(\omega)-\Theta-\sum_{j=1}^m a_j dd^c \log \lvert s_j\rvert_{h_j}^2=\Ric(\tilde{\omega})-\sum_{j=1}^m a_j dd^c \log \lvert s_j\rvert_{h_j}^2.
\end{equation}
Set $f:=\sum_{j=1}^m a_j \, \log \lvert s_j\rvert_{h_j}^2$. Then $\Theta+dd^c f=\sum_{j=1}^m a_j [E_j]$ and $\pi^* \omega^n= e^f \tilde{\omega}^n$. Also, observe that $\pi|_{Y\setminus \cup_j E_j}$ is a biholomorphism.\\

The goal of this section is to prove the following result:

\begin{theorem}
    \label{thm:Convexity_Mabuchi_Anal_Sing} 
 Let $\varphi\in \cE(X,\theta,\phi)$ such that $\theta_\varphi^n= m_\phi\, \omega^n$, let $u_0,u_1\in \PSH(X,\theta)$ with $\phi$-relative minimal singularities and let $(u_t)_{t\in [0,1]}$ be the psh geodesic connecting $u_0$ and $u_1$. Then
 \begin{equation}
     \label{eqn:Formula almost convexity for analytic singularities}
     \cM_{\theta,\varphi}(u_t)\leq t\cM_{\theta,\varphi}(u_1)+(1-t)\cM_{\theta,\varphi}(u_0) +\frac{n\lVert u_0-u_1\rVert_\infty}{2 m_\phi} \{\eta^{n-1}\}\cdot K_{Y/X}
 \end{equation}
 for any $t\in [0,1]$.
\end{theorem}

\medskip

In this first Lemma we show how to go from $(X,\theta,{\phi})$ to $(Y,\eta)$ and back. 
\begin{lemma}\label{lem:Bijection}
There exists a unique map 
$$\mathbf{L}:\PSH(X,\theta, {\phi})\to \PSH(Y,\eta)$$ such that for $u \in \PSH(X,\theta)$ we have 
$$ u \circ \pi  + \tilde{\phi} = \mathbf{L}(u) + \phi \circ \pi$$
where the function $\tilde{\phi}$ is defined in  \eqref{eqn:Almost_Anal_Sing}, (Sometimes for convenience we will simply write 
$\mathbf{L}(u) := (u-{\phi})\circ \pi + \tilde{\phi}$; note however that this equality makes sense only at points where $\phi \circ \pi \neq - \infty$). Moreover:
  \begin{itemize}
        \item[(i)] $\mathbf{L}$ is a bijection;
       
        \item[(ii)]  if $t\to u_t$ is a psh geodesic joining $u_0,u_1\in \mathcal{E}^1(X,\theta,\phi)$ then $t\to v_t:=\mathbf{L} (u_t)$ is a psh geodesic in $\mathcal{E}^1(Y, \eta)$ joining $v_0:=\mathbf{L}(u_0),v_1:=\mathbf{L}(u_1)$;
        
       \item[(iii)] The map $\mathbf{L}$ produces a bijection between $\mathcal{E}(X,\theta,\phi)$ and $\mathcal{E}(Y,\eta)$ (resp. $\mathcal{E}^p(X,\theta,\phi)$ and $\mathcal{E}^p(Y,\eta)$ for any $p\geq 1$); 
       
       \item[(iv)]  ${E}(\theta; u, w)={E}(\eta; \mathbf{L}(u),\mathbf{L}(w) )$ for any $u,w\in\mathcal{E}^1(X,\theta,\phi)$;
        
        \item[(v)]  $\Ent_\theta(u, w)=\Ent_\eta (\mathbf{L}(u), \mathbf{L}(w))$
        for any $u,w\in\mathcal{E}(X,\theta,\phi).$
   \item[(vi)] Let $u_X\in \PSH(X,\theta, \phi)$ and $v_Y:=\mathbf{L}(u_X)$. If the current $\eta_{v_Y}$ has good Ricci curvature then  $\theta_{u_X}$ does too. Furthermore
    \begin{equation}  \label{Poinc-Lel} 
    \Ric(\eta_{v_Y}) = \pi^*(\Ric(\theta_{u_X})) - [K_{Y/X}],
    \end{equation}  
    and
    \begin{equation}
        \label{eqn:Ricci_Energy_UpDown}
        {E}_{\Ric(\theta_{u_X})}(\theta; u, w)={E}_{\Ric(\eta_{v_Y})}\big(\eta; \mathbf{L}(u),\mathbf{L}(w) \big)
    \end{equation}
    for any $u,w$ with $\phi$-relative minimal singularities.
    
    \end{itemize}
\end{lemma}
\noindent In the above statement the energy functionals in (iv) and in \eqref{eqn:Ricci_Energy_UpDown} are defined in \eqref{eq: energy gen down} and \eqref{eq: energy ricci down}.

\begin{proof}
The proof of the first three points proceeds as in the K\"ahler case (\cite[Lemma 4.6]{Trus20b} and \cite[Proposition 3.10]{Tru1}). We give nevertheless some details for the reader's convenience.  \\
We start proving that $\mathbf{L}$ is well defined.
In fact,  set 
$ \tilde{v} := (u-{\phi})\circ \pi + \tilde{\phi},$
then  by \eqref{eqn:Almost_Anal_Sing} 
$$\eta + dd^c((u-{\phi})\circ \pi + \tilde{\phi})  = \eta_{\tilde{\phi}} + \pi^*\theta_u - \pi^*\theta_{\phi} = \pi^*\theta_u - [F].$$ The above means that the restriction of 
$\eta_{\tilde{v}}$ to $Y \setminus F$ is positive, moreover, since $u$ is more singular than $\phi,$ we infer that $ \tilde{v}$ is bounded from above on $Y\setminus F$. Hence there exists a unique  $\eta$-psh function $v$ on $Y$ which equals $\tilde{v}$ almost everywhere with respect to the Lebesgue measure.
Therefore the quasi-psh functions 
$u \circ \pi  + \tilde{\phi}$ and $v + {\phi}\circ \pi  $ coincide almost everywhere on $Y,$ hence they coincide everywhere. 

By construction $\mathbf{L}$ is clearly injective. We now show the surjectivity. For any $v\in \PSH(Y,\eta)$ we claim that $v+{\phi}\circ \pi-{\tilde{\phi}}$ coincides almost everywhere with a
$\pi^* \theta$-psh. 
Indeed by \eqref{eqn:Almost_Anal_Sing}, $\pi^*\theta+dd^c (v+{\phi}\circ \pi-{\tilde{\phi}})=\eta_v+ [F]$.
Thus, since the fibers of $\pi$ are connected and compact, there exists $u\in \PSH(X,\theta)$ such that $u\circ \pi:= v+{\phi}\circ \pi-{\tilde{\phi}}$  almost everywhere on $X$ (see for instance \cite[Proposition 1.2.7.(ii)]{BouThesis}), Using the same arguments as above we see that we must have $\mathbf{L}(u) = v.$ This concludes the proof of $(i)$.\newline
As the non-pluripolar product does not put mass on pluripolar sets, we have
\begin{equation}\label{eq blow up}
    \pi^*(\theta_u^j \wedge \theta_w^{n-j})=(\pi^*\theta_u)^j\wedge (\pi^*\theta_w)^{n-j}=\eta_{\mathbf{L}(u)}^j\wedge \eta_{\mathbf{L}(w)}^{n-j}
\end{equation}
for any $j=0,\dots,n$ and any $u\in\PSH(X,\theta,{\phi})$. Thus
$$\int_X (u-w) \theta_u^j \wedge \theta_w^{n-j} =\int_{Y\setminus F} ((u-{\phi})\circ \pi{+\tilde{\phi}} - (w-{\phi})\circ \pi{-\tilde{\phi}}) \,\eta_{\mathbf{L}(u)}^j\wedge \eta_{\mathbf{L}(w)}^{n-j}= \int_{Y} ( \mathbf{L}(u) -\mathbf{L}(w)) \,\eta_{\mathbf{L}(u)}^j\wedge \eta_{\mathbf{L}(w)}^{n-j}.
$$
Similarly, for any $p\geq 1$ we have
    $$
    \int_X \left\lvert u-\phi \right\rvert^p \theta_u^n= \int_Y \left\lvert \mathbf{L}(u)-\mathbf{L}(\phi) \right\rvert^p \eta_{\mathbf{L}(u)}^n.
    $$
    
    Then (iii) and (iv) follow.

    Let us prove (ii). Let $p_X: X\times A\to X$ and $p_Y:Y\times  A \to Y$ be the projections on the first factors, and consider
    {\small
    $$
    U(x,z):=u_{\log |z|}(x)\in \PSH\big(X\times A,p_X^*\theta\big), \, \,
    V( y,z):=v_{\log |z|}(y)=(u_{\log|z|}(x) -{\phi}(x))\circ \pi{+\tilde{\phi}(y)} \in  \PSH\big(Y\times A,p_Y^*\eta ).
    $$
    }
    Then we find:
    \begin{equation}
        \label{eqn:Geod_Up_Down}
        (\pi\times \text{Id})^*\big(p_X^*\theta+dd^c_{(x,z)} U\big)=p_Y^*\eta+dd^c_{(y,z)} V+p_Y^*[F].
    \end{equation}
   On the other hand, by assumption there exists $M>0$ such that $\max(u_0,u_1) \leq {\phi}+ M$; then by convexity $u_{\log |z|} \leq {\phi} + M$ for all $z \in A.$ It follows that the function $V$ is bounded from above and it is $p_Y^*\eta$-psh on $Y \times A \setminus p_Y^{-1}(F).$
 Then $V$ extends to an $p_Y^*\eta$-psh function on $Y\times A$, i.e. $t\to v_t$ (with $t = \log|z|$) is a psh subgeodesic. Note that the argument above says that $\mathbf{L}$ produces a injective map from psh subgeodesics joining $u_0,u_1$ and psh subgeodesics joining $v_0,v_1$. Such map is actually a bijection as the surjectivity follows reading (\ref{eqn:Geod_Up_Down}) backwards,  using the fact that $F$ is effective and taking the pushforward by $\pi \times \text{Id}.$ Moreover, this correspondence between psh subgeodesics respects the partial order $\leq$. Namely, for any couple of psh subgeodesics $U_1,{U}_2$ joining $u_0,u_1$ such that ${U}_1\leq {U}_2$, the corresponding psh subgeodesics ${V}_1(y,z)= (\pi\times \text{Id})^* (U_1(x,t)-{\phi}(x)){+\tilde{\phi}(y)},{V}_2 (y,z)=(\pi\times \text{Id})^* (U_2(x,t)-{\phi}(x)){+\tilde{\phi}(y)}$ clearly satisfy ${V}_1\leq {V}_2$, and vice-versa. The proof of (ii) follows from the maximality of psh geodesics.

\medskip
As seen above, for any $\theta$-psh function $u\in \PSH(X,\theta,{\phi})$ we have $ \pi^* \theta_u^n
= \eta_{\mathbf{L}(u)}^n $.  
Also, we already observed that $m_{\phi}=\vol(\eta)$. Then for any $u,w\in \mathcal{E}(X, \theta, \phi)$ we have $m_u=m_w=m_\phi=\vol(\eta)$. The entropy formula in (v) then follows from the first item of Lemma \ref{lem:Holo}.

\smallskip

   We now prove (vi). Let $p \in Y,$ and  $\Omega$ a nowhere zero holomorphic section of $K_Y$ near $p,$ and $\Omega'$ a nowhere zero local holomorphic section of $K_X$ near $\pi(p).$ Note that $\pi^*\Omega'= h \Omega$ for some holomorphic function vanishing on each $E_j$.
    If $\eta_{v_Y}^n = g \ i^{n^2} \Omega\wedge \bar{\Omega}$ for some function $g\geq 0,$ then on $X\setminus \pi(\cup_j E_j)$
     \begin{equation}\label{Ricci on the pullback}
       \pi_*  \eta_{v_Y}^n = \theta_{u_X}^n = (g \circ \pi^{-1})\,  i^{n^2} \pi_*(\Omega \wedge \overline{\Omega}) =  (g \circ \pi^{-1})\, |h\circ \pi^{-1}|^{-2} \, i^{n^2}  \Omega' \wedge \overline{\Omega'}.
     \end{equation}
We then observe that $\log(g \circ \pi^{-1})\in L^1(\omega^n)$ since 
$$\int_{X } |\log(g \circ \pi^{-1})|\, \omega^n =\int_{Y } |\log g| \,e^f \, \tilde{\omega}^n $$
and the latter integral is finite since $\log g\in L^1(\tilde{\omega}^n)$ and $e^f$ is bounded.

    Therefore $\theta_{u_X}$ has well defined Ricci if so does $\eta_{v_Y}$. Moreover, from the identity in \eqref{Ricci on the pullback} we find that on $X\setminus \pi(\cup_j E_j)$
    $$\Ric(\theta_{u_X})= \pi_*\Ric(\eta_{v_Y}) + \pi_* dd^c \log|h|^2.$$
    Note that, since $dd^c \log|h|^2=0$ on $Y\setminus \cup_j E_j$ and $\pi$ is a biholomorphism there, we obtain that $\pi^*\Ric(\theta_{u_X})-\Ric(\eta_{v_Y})=0$ on $Y\setminus \cup_j E_j $, or equivalently $\pi^*\Ric(\theta_{u_X})-\Ric(\eta_{v_Y})$ is a $(1,1)$-current supported on the singular locus.
    Since $K_Y= \pi^* K_X +  \sum_j a_j E_j, \pi^*\{\Ric(\theta_{u_X}) \}=-\pi^*c_1(K_X)$ and $\{\Ric(\eta_{v_Y}) \}=-c_1(K_Y) $ we obtain \eqref{Poinc-Lel}. Observe that the last claim holds since the information on the cohomology transfers at the level of forms since $E_j$ are numerically independent.
  Indeed by Hironaka \cite{Hir64} any modification $\pi:Y\to X$ can be dominated by a map $p:Z\to X$ given by a sequence of blow-ups along smooth centers, i.e. there exists $\tau:Z\to Y$ such that $p=\pi\circ \tau$. In particular any linear combination of (classes of) $\pi$-exceptional divisors becomes a linear combination of $p$-exceptional divisors after pulling-back through $\tau$, thus the numerical independence follows from \cite[Page 605]{GH}.\\
   Also, since the non-pluripolar product does not put mass on the pluripolar sets, we deduce that for any $k=0,\cdots ,n-1$
     \begin{equation*}
      \int_X (u-w)\,\theta_u^k\wedge \theta_w^{n-k-1} \wedge \Ric(\theta_{u_X}) =\int_Y (\mathbf{L}(u)-\mathbf{L}(w))\,  \eta_{\mathbf{L}(u)}^k \wedge \eta_{\mathbf{L}(w)}^{n-k-1}\wedge \Ric(\eta_{v_Y})
    \end{equation*}
    for any $u,w\in\mathcal{E}(X,\theta,\phi)$ with $\phi$-relative minimal singularities. This yields \eqref{eqn:Ricci_Energy_UpDown} and concludes the proof.
\end{proof}

\begin{remark}\label{rk:ricci}
It is worth to mention that given two positive real closed $(1,1)$-currents $T$ (on $X$) and $S$ (on $Y$) whose cohomology classes are big and such that $\pi^* T^n= S^n$, the same arguments in the above lemma ensure that $T$ has good Ricci curvature if so does $S$.
\end{remark}

\begin{prop}
\label{cor:Mabuchi_UpDown}
    Let $\varphi\in \cE(X,\theta,\phi)$ such that $\theta_\varphi^n=m_\phi \, \omega^n$, and let $\eta_w$ be such that $\eta_w^n=m_\phi \, \tilde{\omega}^n$. Using the same notations of Lemma \ref{lem:Bijection}, we set $\hat{\varphi}:=\mathbf{L}(\varphi)$ and $v:=\mathbf{L}(u)$ for $u\in \cE(X,\theta,
    \phi)$ with $\phi$-relative minimal singularities. If $\Ent_\theta(u)<+\infty$, then
    \begin{equation}\label{eqn:Mabuchi for gentle analytic singularities}
        \cM_{\theta,\varphi}(u)=\cM_{\eta,w}(v)-\cM_{\eta,w}(\hat{\varphi})+\frac{n}{\vol(\eta)}\sum_{E_j\not\subset E_{nK}(\eta) } a_j \vol(\eta_{|E_j})\Big\{E(\eta; v,\hat{\varphi})- E_{E_j}(\eta|_{E_j}; {v}_{|{E_j}}, \hat{\varphi}_{|E_j})\Big\}
    \end{equation}
    where
    $$
    E_{E_j}(\eta|_{E_j}; {v}_{|{E_j}}, \hat{\varphi}):=
        \frac{1}{n \vol(\eta_{|E_j})}\sum_{k=0}^{n-1}\int_{E_j} (v-\hat{\varphi})\, \eta_{v}^k \wedge \eta_{\hat{\varphi}}^{n-k-1} 
    $$
    is the energy relative to the smooth submanifold $E_j$.\\
    Here $\cM_{\theta,\varphi}$ and $\cM_{\eta,w}$ denote the Mabuchi functionals relative to $\varphi, \hat{\varphi}$ defined in (\ref{eq: K-energy}).
\end{prop} 
Observe that $\cM_{\eta,w}(\hat{\varphi})<+\infty$ since $\eta_{\hat{\varphi}}^n=\pi^*\omega^n= e^f \tilde{\omega}^n$ has $L^p$ density for $p>1$ with respect to $\eta_w^n=\vol(\eta)\, \tilde{\omega}^n$.
\smallskip

It is crucial to stress  that $\int_{E_j}\eta^{n-1}>0$ if and only if $E_j\not\subset E_{nK}(\eta)$ by the main result in \cite{CT15}. In particular for such $j$, $\{\eta_{|E_j}\}$ is a big and nef class on $E_j$ and $\vol(\eta_{|E_j})=\int_{E_j}\eta^{n-1}$. Moreover, in this case any $v\in \PSH(Y,\eta)$ with minimal singularities restricts to a function $v_{|E_j}\in \PSH(E_j,\eta_{|E_j})$ with full Monge-Ampère mass as a consequence of the following Lemma.

\begin{lemma}\label{lem:Restr of min sing}
    Let $\eta$ be a smooth and closed $(1,1)$-form representing a big and nef class and let $v\in \PSH(Y,\eta)$ with minimal singularities. Assume also that $Z\subset Y$ is a positive dimensional
     connected submanifold (we also allow $Z=Y$) such that $Z\not \subset E_{nK}(\eta)$ and let $\Gamma$ be a semipositive smooth and closed $(p,p)$-form, $0\leq p\leq \dim Z$. Then $v_{|Z}\in \PSH(Z,\eta_{|Z})$ satisfies
    $$
    \int_Z \langle \Gamma_{|Z} \wedge (\eta_{|Z}+dd^c v_{|Z})^{\dim Z-p}\rangle=\int_Z \Gamma_{|Z} \wedge \eta_{|Z}^{\dim Z-p},
    $$
     where at the LHS we have the non-pluripolar product on $Z$, while at the RHS we consider the usual wedge product between smooth forms on $Z$.
\end{lemma}

\begin{proof}
To lighten notations, we merely write $\langle \Gamma \wedge (\eta+dd^c v)^{\dim Z-p}\rangle $ instead of $\langle \Gamma_{|Z} \wedge (\eta_{|Z}+dd^c v_{|Z})^{\dim Z-p}\rangle$. More generally, we drop the notation for the restriction over $Z$.\\
    By \cite[Theorem 3.17]{Bou04} there exists a function $\psi\in \PSH(Y,\eta)$ with analytic singularities along the non-Kähler locus $E_{nK}(\eta)$ such that $T:=\eta+dd^c\psi\geq \varepsilon \tilde{\omega} $ where $\tilde{\omega}$ is a fixed Kähler form on $Y$. 
    By nefness of $\{\eta\}$, for any $\delta>0$ there exists also a K\"ahler form $\eta+\delta\tilde{\omega}+dd^c \varphi_\delta$. Then we set $u_\delta:=\frac{\varepsilon}{\delta+\varepsilon} \varphi_\delta+ \frac{\delta}{\delta+\varepsilon}\psi$. Such a function is $\eta$-psh as
    $$
    \eta+dd^c u_\delta= \frac{\varepsilon}{\delta+\varepsilon}(\eta+\delta\tilde{\omega}+dd^c \varphi_\delta)+\frac{\delta}{\delta+\varepsilon}(\eta-\varepsilon\tilde{\omega}+dd^c \psi)\geq 0.
    $$
    Since by assumption $Z\not\subset E_{nK}(\eta)$, the function $u_{\delta|Z}$ is a well-defined $\eta_{|Z}$-psh function.
    Thus as $v\in \PSH(Y,\eta)$ has minimal singularities we have $v_{|Z}\geq u_{\delta|Z}-C$ and from \cite[Theorem 1.1]{DDL2} it follows that
    \begin{eqnarray*}
        \int_Z \langle\Gamma \wedge(\eta+dd^c v)^{\dim Z-p}\rangle &\geq & \int_Z\langle \Gamma \wedge(\eta+dd^c u_{\delta})^{\dim Z-p}\rangle \\
        &\geq & \left(\frac{\varepsilon}{\delta+\varepsilon}\right)^{\dim Z-p}\int_Z\Gamma \wedge (\eta+\delta\tilde{\omega}+dd^c \varphi_{\delta})^{\dim Z-p}\\
        &\geq &  \left(\frac{\varepsilon}{\delta+\varepsilon}\right)^{\dim Z-p}\int_Z \Gamma \wedge \eta^{\dim Z-p}
    \end{eqnarray*}
    where the last equality follows from Stokes' theorem and the positivity of $\tilde{\omega}$. Letting $\delta\to 0$, we find that 
    $$\int_Z\langle\Gamma\wedge (\eta+dd^c v)^{\dim Z-p}\rangle\geq \int_Z \Gamma \wedge\eta^{\dim Z-p}.$$ For the reverse inequality we observe that  for any $\varepsilon > 0,$ 
    $\varphi_\delta$ is with minimal singularities in 
    $\eta  + \varepsilon \tilde{\omega}$ so, again from \cite[Theorem 1.1]{DDL2} and Stokes' Theorem, it follows that
    \begin{eqnarray*}
    \int_Z \langle\Gamma\wedge (\eta+dd^c v)^{\dim Z-p}\rangle &\leq & \int_Z \langle \Gamma \wedge (\eta+\varepsilon\tilde{\omega}+dd^c v)^{\dim Z-p} \rangle \\
    &\leq  & \int_Z \Gamma \wedge (\eta+\varepsilon\tilde{\omega}+dd^c \varphi_\delta)^{\dim Z-p} \\
      &=  & \int_Z \Gamma \wedge (\eta+\varepsilon\tilde{\omega})^{\dim Z-p} \\
    \end{eqnarray*}
Letting $\varepsilon\to 0$ then concludes the proof.

\end{proof}

We are now ready to prove Proposition \ref{cor:Mabuchi_UpDown}.
\begin{proof}[Proof of Proposition \ref{cor:Mabuchi_UpDown}]
We first assume that $\theta_u^n$ has bounded density, i.e. $\theta_u^n=g\omega^n$ with $g$ bounded.

\noindent \textbf{Step 1: A first formula connecting $\cM_{\theta,\varphi}(u)$ and $\cM_{\eta,w}(v)$.}\\
Lemma \ref{lem:Bijection} gives $$ E(\theta; u,\varphi) =E(\eta; v, \hat{\varphi}).$$
    Moreover, as $\pi^*\Ric(\omega)=\Ric(\tilde{\omega})+\Theta$ by \eqref{eqn:pullback of Ricci}, $\Ric(\omega)=\Ric(\theta_\varphi), \Ric(\tilde{\omega})=\Ric(\eta_w)$ we similarly have
    \begin{eqnarray*}
    E_{\Ric(\theta_\varphi)}(\theta; u,\varphi)&=& E_{\pi^*\Ric(\omega)}(\eta; v,\hat{\varphi})\\
     &=& E_{\Ric(\tilde{\omega})}(\eta; v,\hat{\varphi})+E_{\Theta}(\eta; v,\hat{\varphi})\\
    &=& E_{\Ric(\eta_w)}(\eta; v,\hat{\varphi})+E_{\Theta}(\eta; v,\hat{\varphi}).
    \end{eqnarray*}
   Furthermore, since $\eta_{\hat{\varphi}}^n=m_\phi e^f\tilde{\omega}^n$ for $f=\sum_{j=1}^m a_j \log \lvert s_j\rvert_{h_j}^2$, by Lemma \ref{lem:Bijection} and the definition of the entropy we have
    $$
    \Ent_\theta(u,\varphi)=\Ent_\eta(v,\hat{\varphi})=\Ent_\eta(v)-\frac{1}{\vol(\eta)}\int_Y f\, \eta_v^n.
    $$
   Combining all the above and using the cocycle property for the energies (e.g. $E(\eta; v,\hat{\varphi})=E(\eta; v,w)-E(\eta;\hat{\varphi},w)$), we obtain
    {\small
    \begin{multline}
        \label{eqn:Almost Formula Mabuchi pullback}
        \cM_{\theta,\varphi}(u)= \bar{S}_{{\varphi}}E(\eta; v,\hat{\varphi})-n E_{\Ric(\eta_w)}(\eta; v,\hat{\varphi})-nE_{\Theta}(\eta; v,\hat{\varphi})+\Ent_\eta(v)-\frac{1}{\vol(\eta)}\int_Y f\, \eta_v^n\\
        =\cM_{\eta,w}(v)-\cM_{\eta,w}(\hat{\varphi})+(\bar{S}_\varphi-\bar{S}_w)E(\eta;v,\hat{\varphi})-nE_\Theta(\eta;v,\hat{\varphi})+\Ent_\eta(\hat{\varphi})-\frac{1}{\vol(\eta)}\int_Y f\, \eta_v^n.
    \end{multline}
    } 
    
  \smallskip

\noindent \textbf{Step 2: The twisted energy $E_\Theta(\eta;v,\hat{\varphi})$.}\\
   Using the same ideas in the proof of \cite[Theorem 4.2]{DNL21}, we  approximate $v, \hat{\varphi}$ by decreasing sequences $v_j,\hat{\varphi}_j$ of bounded $(\eta+\varepsilon_j \tilde{\omega})$-psh functions such that the respective entropy converges and $v_j-\hat{\varphi}_j$ is uniformly bounded.\\
    Observe that $\eta_{\hat{\varphi}}^n=m_\phi e^f\tilde{\omega}^n$ and $\eta_v^n=\pi^* \theta_u^n= (g\circ \pi) \pi^*\omega^n= g\circ \pi e^f \tilde{\omega}^n $. In particular, the measures $\eta_{\hat{\varphi}}^n, \eta_v^n$ both have bounded density. \\
    We set $\eta_j:=\eta+\frac{1}{j}\tilde{\omega}$, and let $\alpha>0$ be small enough so that $\sup_{w\in \PSH(Y,\eta)}\int_Y e^{-2\alpha(w-\sup_Y w)}\tilde{\omega}^n<\infty$. We then define $v_j\in\cE(Y,\eta_j), \hat{\varphi}_j\in \cE(Y,\eta_j)$ as the unique bounded solutions of
    \begin{equation}
        \label{eqn:Trick}
        (\eta_j+dd^c v_j)^n= e^{\alpha(v_j-v)} (\eta+dd^c v)^n= e^{\alpha(v_j-v)} g\circ \pi e^f \tilde{\omega}^n  
    \end{equation}
    $$(\eta_j+dd^c \hat{\varphi}_j)^n=e^{\alpha(\hat{\varphi}_j- \hat{\varphi})}(\eta+dd^c \hat{\varphi})^n=e^{\alpha(\hat{\varphi}_j - \hat{\varphi}) }m_\phi e^f\tilde{\omega}^n.$$
   
     Note that, since $\{\eta_j\}$ is a K\"{a}hler class, the existence of bounded  $v_j,\hat{\varphi}_j$ follows from \cite{BEGZ10} since $ e^{-\alpha(v-\sup_X v)}  g\circ \pi e^f $ and $e^{-\alpha( \hat{\varphi}-\sup_X \hat{\varphi}) } e^f$ are in $L^2$.

    From the comparison principle (see for instance \cite[Lemma 2.5]{DDL1}) we obtain that $v_j, \hat{\varphi}_j$ are decreasing sequences converging respectively to $v$ and $\hat{\varphi}$. Moreover \cite[Theorem 1.9]{DNGG20} implies that $v_j$ and $\hat{\varphi}_j$ are uniformly bounded, from which we deduce that $v_j-\hat{\varphi}_j$ is uniformly bounded.
    
As the functions $v_j, \hat{\varphi}_j$ are bounded, and $v,\hat{\varphi}$ have minimal singularities, it follows from Lemma \ref{lem:Restr of min sing} that for any smooth closed semipositive $(1,1)$-form $\Gamma$ and for any $k=0,\dots,n-1$ we have
    \begin{equation}
        \label{eqn:Masses}
        \int_Y \langle \Gamma \wedge \eta_{j,v_j}^k\wedge \eta_{j,\hat{\varphi}_j}^{n-k-1}\rangle=\int_Y \Gamma \wedge \eta_j^{n-1} \longrightarrow \int_Y \Gamma \wedge \eta^{n-1}=\int_Y\langle \Gamma \wedge \eta_v^k\wedge \eta_{\hat{\varphi}}^{n-k-1} \rangle,
    \end{equation}
   as  ${j\to +\infty}$. In the above, we emphasized by $\langle\cdots \rangle$ the use of the non-pluripolar product vs the usual wedge product between smooth forms.\\
    Thus, writing $\Theta$ as difference of two semipositive closed smooth $(1,1)$-forms it follows from Theorem \ref{thm:weak convergence} that for any $k=0,\dots,n-1$,
    \begin{equation}\label{eqn:Useful eqn 1}
    \int_Y (v_j-\hat{\varphi}_j) \Theta\wedge \eta_{j,v_j}^k\wedge \eta_{j,\hat{\varphi}_j}^{n-k-1} \longrightarrow \int_Y (v-\hat{\varphi}) \Theta\wedge \eta_v^k\wedge \eta_{\hat{\varphi}}^{n-k-1}   
    \end{equation}
    as $j\to +\infty$. \\
    It also follows from Lemma \ref{lem:Restr of min sing} that if $E_l\not\subset E_{nK}(\eta)$, for any $k=0,\dots,n-1$ we have
    $$
    \int_{E_l}\langle\eta_{j,v_j}^k\wedge \eta_{j,\hat{\varphi}_j}^{n-k-1}\rangle=\int_{E_l}\eta_j^{n-1}\longrightarrow\int_{E_l}\eta^{n-1}=\int_{E_l}\langle \eta_v^k\wedge \eta_{\hat{\varphi}}^{n-k-1} \rangle.
    $$
 Therefore, Theorem \ref{thm:weak convergence} ensures that for any $E_l\not \subset E_{nK}(\eta)$ and any $k=0,\dots,n-1$
    \begin{equation}\label{eqn:Useful eqn 2}
        \int_{E_l}(v_j-\hat{\varphi}_j)\eta_{j,v_j}^k\wedge \eta_{j,\hat{\varphi}_j}^{n-k-1}\longrightarrow\int_{E_l}(v-\hat{\varphi})\eta_v^k\wedge \eta_{\hat{\varphi}}^{n-k-1}
    \end{equation}
    as $j\to +\infty$. If instead $E_l\subset E_{nK}(\eta)$ then
    \begin{equation}\label{eqn:Useful eqn 3}
        \left|\int_{E_l}(v_j-\hat{\varphi}_j)\eta_{j,v_j}^k\wedge \eta_{j,\hat{\varphi}_j}^{n-k-1}\right|\leq \lVert v_j-\hat{\varphi}_j\rVert_\infty \vol(\eta_{j|E_l})\longrightarrow 0
    \end{equation}
    as $j\to +\infty$ since $ v_j-\hat{\varphi}_j$ is uniformly bounded while $\vol(\eta_{j|E_l})\to 0$. The last claim is a consequence of the fact that $E_{nK}(\eta)=\bigcup\left\{V\subset X \, : \, \int_V\eta^{\dim V}=0\right\}$ by \cite{CT15}.

\smallskip

    Next, we recall that for any $j$, as $v_j,\hat{\varphi}_j$ are bounded, the non-pluripolar product coincides with the Bedford-Taylor wedge product. So, since $\Theta+dd^cf=[K_{Y/X}]=\sum_{l=1}^m a_l[E_l]$, we obtain
    \begin{equation}\label{eqn:Useful eqn 4}
        \langle\Theta\wedge \eta_{j,v_j}^k\wedge \eta_{j,\hat{\varphi}_j}^{n-k-1}\rangle=\Theta\wedge \eta_{j,v_j}^k\wedge \eta_{j,\hat{\varphi}_j}^{n-k-1}= \sum_{l=1}^m a_l [E_l]\wedge \eta_{j,v_j}^k\wedge \eta_{j,\hat{\varphi}_j}^{n-k-1} - dd^c f \wedge \eta_{j,v_j}^k\wedge \eta_{j,\hat{\varphi}_j}^{n-k-1},
    \end{equation}
    where the products $[E_l]\wedge \eta_{j,v_j}^k\wedge \eta_{j,\hat{\varphi}_j}^{n-k-1}, dd^c f \wedge \eta_{j,v_j}^k\wedge \eta_{j,\hat{\varphi}_j}^{n-k-1} $ make sense thanks to \cite[Section 2]{Dem85}.
    Observe also that by definition we have
    $$
    dd^c f \wedge \eta_{j,v_j}^k\wedge \eta_{j,\hat{\varphi}_j}^{n-k-1}:=dd^c\left(f\,\eta_{j,v_j}^k\wedge \eta_{j,\hat{\varphi}_j}^{n-k-1}\right).
    $$
    The above is a $(n,n)$-current which acts on smooth functions. However \cite[Theorem 2.2]{Dem85} implies that such action can be extended to {bounded} quasi-psh functions. It then follows (basically) by definition that 
    \begin{eqnarray}\label{eqn:Useful eqn 5}
    \nonumber    \int_X (v_j-\hat{\varphi}_j)dd^c f \wedge \eta_{j,v_j}^k\wedge \eta_{j,\hat{\varphi}_j}^{n-k-1}&=& \int_X fdd^c (v_j-\hat{\varphi}_j) \wedge \eta_{j,v_j}^k\wedge \eta_{j,\hat{\varphi}_j}^{n-k-1}\\
        &=&\int_X f \eta_{j,v_j}^{k+1}\wedge \eta_{j,\hat{\varphi}_j}^{n-k-1} -\int_X f \eta_{j,v_j}^{k}\wedge \eta_{j,\hat{\varphi}_j}^{n-k}.
    \end{eqnarray}

    Finally, by \eqref{eqn:Trick} and Monotone Convergence Theorem
    \begin{equation}\label{eqn:Useful eqn 6}
    \int_Y f\eta_{j,v_j}^n= \int_Y f e^{\alpha(v_j-v)} \eta_v^n\longrightarrow \int_Y f\eta_v^n
    \end{equation}
    and
   $$ \int_Y f \eta_{j,\hat{\varphi}_j}^n = \int_Y f e^{\alpha(\hat{\varphi}_j-\hat{\varphi})} \eta_{\hat{\varphi}}^n \longrightarrow \int_Y f \eta_{\hat{\varphi}}^n. $$
 Note also that the integrals at the RHS are finite since $\eta_v^n,\eta_{\hat{\varphi}}$ have bounded density and $f\in L^1(\tilde{\omega}^n).$\\
 Combining all the above we have
    {\small
    \begin{eqnarray*}
        n\vol(\eta)E_{\Theta}(\eta; v,\hat{\varphi})&\overset{\eqref{eqn:Useful eqn 1}}{=} &\lim_{j\to +\infty}\sum_{k=0}^{n-1}\int_X(v_j-\hat{\varphi}_j)\langle\Theta\wedge \eta_{j,v_j}^k\wedge \eta_{j,\hat{\varphi}_j}^{n-k-1}\rangle\\
        &\overset{\eqref{eqn:Useful eqn 4}}{=} &\lim_{j\to+\infty}\left(\sum_{l=1}^m a_l\sum_{k=0}^{n-1} \int_{E_l} (v_j-\hat{\varphi}_j) \eta_{j,v_j}^k \wedge \eta_{j,\hat{\varphi}_j}^{n-k-1}-\sum_{k=0}^{n-1} \int_Y (v_j-\hat{\varphi}_j) dd^c f\wedge \eta_{j,v_j}^k \wedge \eta_{j,\hat{\varphi}_j}^{n-k-1}\right)\\
        &\overset{\eqref{eqn:Useful eqn 5}}{=} & \lim_{j\to +\infty}\left(\sum_{l=1}^m a_l\sum_{k=0}^{n-1} \int_{E_l} (v_j-\hat{\varphi}_j) \eta_{j,v_j}^k \wedge \eta_{j,\hat{\varphi}_j}^{n-k-1}-\int_Y f \,\eta_{j,v_j}^n+\int_Y f \,\eta_{j,\hat{\varphi}_j}^n\right)\\
        &{=} &\sum_{E_j\not\subset E_{nK}(\eta)}a_l\sum_{k=0}^{n-1} \int_{E_l} (v-\hat{\varphi}) \eta_v^k \wedge \eta_{\hat{\varphi}}^{n-k-1}-\int_Y f \,\eta_v^n+\int_Y f \,\eta_{\hat{\varphi}}^n\\
        &= &\sum_{E_j\not\subset E_{nK}(\eta)}n a_l \vol(\eta_{|E_l}) E_{|E_l}(\eta_{|E_l}; v_{|E_l},\hat{\varphi}_{|E_l})-\int_Y f\, \eta_v^n+\vol(\eta)\Ent_\eta(\hat{\varphi}).
    \end{eqnarray*}
}
where in the fourth equality we used \eqref{eqn:Useful eqn 2}, \eqref{eqn:Useful eqn 3} and \eqref{eqn:Useful eqn 6}; in the last equality we also used that $\Ent_\eta(\hat{\varphi})=\frac{1}{\vol(\eta)}\int_Y f\, \eta_{\hat{\varphi}}^n$.\\
Thus, \eqref{eqn:Almost Formula Mabuchi pullback} writes as
  {\small
    \begin{equation}\label{eqn:Almost the end}
        \cM_{\theta,\varphi}(u)=\cM_{\eta,w}(v)-\cM_{\eta,w}(\hat{\varphi})+\big(\bar{S}_\varphi-\bar{S}_w\big)E(\eta;v,\hat{\varphi})-\frac{n}{\vol(\eta)}\sum_{j=1}^m a_j \vol(\eta_{|E_j}) E_{|E_j}(\eta_{|E_j}; v_{|E_j},\hat{\varphi}_{|E_j}).
    \end{equation}}
    
\noindent \textbf{Step 3: Computing $\bar{S}_{\varphi}-\bar{S}_w$}.\\
    Since $\pi^* \Ric(\omega)=\Ric(\tilde{\omega})+\Theta$, by linearity, the proof of Lemma \ref{lem:Bijection} and again Lemma \ref{lem:Restr of min sing} we have
    \begin{align*}
        \bar{S}_\varphi&=\frac{n}{m_\phi}\int_X \langle\Ric(\omega)\wedge \theta_\varphi^{n-1}\rangle\\
        &=\frac{n}{\vol(\eta)}\int_Y \langle\big(\Ric(\tilde{\omega})+\Theta\big)\wedge \eta_{\hat{\varphi}}^{n-1}\rangle\\
        &=\frac{n}{\vol(\eta)}\int_Y\langle \Ric(\tilde{\omega})\wedge \eta_w^{n-1}\rangle+\frac{n}{\vol(\eta)}\int_Y \Theta\wedge \eta^{n-1}\\
        &= \bar{S}_w+\frac{n}{\vol(\eta)} \{\Theta\}\cdot \{\eta^{n-1}\}\\
        &= \bar{S}_w+ \frac{n}{\vol(\eta)} \sum_{E_l \not\subset E_{nK}(\eta)}^m a_l \vol(\eta_{|E_l})
    \end{align*}
    where in the above we used several times that $\hat{\varphi}$ and $w$ have minimal singularities and $\{\Theta\}=\sum_{l=1}^m a_l\{E_l\}$. \\
    Plugging this into \eqref{eqn:Almost the end} concludes the proof for $u\in \cE(X,\theta,\phi)$ with $\phi$-relative minimal singularities such that $\theta_u^n$ has bounded density with respect to $\omega^n$.

\smallskip
    
\noindent \textbf{Step 4: General case}. \\
Let $u\in \cE(X,\theta,\phi)$ with $\phi$-relative minimal singularities such that $\Ent_\theta(u)<+\infty$ and let $g\geq 0$ such that $\theta_u^n=g\omega^n$. We fix $\alpha>0$ small enough so that $e^{-\alpha u}\in L^2(\omega^n)$ and we define $u_k\in \cE(X,\theta,\phi)$ to be the solution of
$$
\theta_{u_k}^n= e^{\alpha(u_k-u)}\min(g,k) \omega^n
$$
(see \cite[Theorem 1.4]{DDL2} for the existence of such potentials). Observe that $u_k$ has $\phi$-relative minimal potential by \cite[Theorem A]{DDL4}.

Set $v_k:=\mathbf{L}(u_k)$. Our goal is to prove that for any $E_j\not \subset E_{nK}(\eta)$ we have
\begin{gather*}
\cM_{\theta,\varphi}(u_k)\to \cM_{\theta,\varphi}(u),\\
\cM_{\eta,w}(v_k)\to \cM_{\eta,w}(v),\\
E(\eta; v_k,\hat{\varphi})\to E(\eta; v,\hat{\varphi}),\\
E_{E_j}(\eta|_{E_j}; {v_k}_{|{E_j}}, \hat{\varphi}_{|E_j})\to E_{E_j}(\eta|_{E_j}; {v}_{|{E_j}}, \hat{\varphi}_{|E_j})
\end{gather*}
 as $k\to +\infty$. By Lemma \ref{lem:Bijection} we have
\begin{equation}\label{eq v}
\eta_{v_k}^n=e^{\alpha(v_k-v)}\min(g\circ \pi,k)e^f\tilde{\omega}^n
\end{equation}
and $v_k\in \cE(Y,\eta)$. Then it follows from the comparison principle in \cite[Lemma 2.5]{DDL1} that $v_k \searrow \tilde{v}$ and $\tilde{v} \geq v$. We claim that $\tilde{v}= v$. This is indeed the case since from \eqref{eq v} we find that $$\eta_{\tilde{v}}^n=e^{\alpha(\tilde{v}-v)} g\circ \pi e^f\tilde{\omega}^n= e^{\alpha(\tilde{v}-v)} \eta_v^n.$$
The comparison principle once again gives $\tilde{v}=v$.
By construction, the previous fact is equivalent to $u_k\searrow u$. As $u$ has $\phi$-relative minimal singularities, we deduce that any difference $u_k-\tilde{u}$ for $\tilde{u}\in \cE(X,\theta,\phi)$ with $\phi$-relative minimal singularities is uniformly bounded in $k$, and the analogous holds for $v_k-\tilde{v}$, $\tilde{v}=\mathbf{L}(u)$. In particular, since $\varphi$ has $\phi$-relative minimal singularities and since $\hat{\varphi}, w$ have minimal singularities, combining Lemma \ref{lem:Restr of min sing} with Theorem \ref{thm:weak convergence} we infer the following convergences of energies:
\begin{gather*}
    E(\theta;u_k,\varphi)\longrightarrow E(\theta;u,\varphi),\quad   E(\eta;v_k,\hat{\varphi})\longrightarrow E(\eta;v,\hat{\varphi}),\quad  E(\eta;v_k,w)\longrightarrow E(\eta;v,w), \\
    E_{\Ric(\theta_\varphi)}(\theta;u_k,\varphi)\longrightarrow E_{\Ric(\theta_\varphi)}(\theta;u,\varphi), \quad 
    E_{\Ric(\eta_w)}(\eta;v_k,w)\longrightarrow E_{\Ric(\eta_w)}(\eta;v,w)\\
    E_{E_j}(\eta_{|E_j};v_{k|E_j},\hat{\varphi}_{|E_j})\longrightarrow E_{E_j}(\eta_{|E_j};v_{|E_j},\hat{\varphi}_{|E_j})
\end{gather*}
for any $E_j\not \subset E_{nK}(\eta)$. Note that we also used that $\Ric(\theta_\varphi)=\Ric(\omega), \Ric(\eta_w)=\Ric(\tilde{\omega})$ are smooth. It remains to prove that
$$
\Ent_\theta(u_k)\to \Ent_\theta(u), \quad \quad \Ent_\eta(v_k)\to \Ent_\eta(v).
$$
In order to do so, we observe that if $0\leq h\leq g$, an elementary calculation gives
\begin{eqnarray*}
    \lvert h\log h\rvert&\leq &\mathbf{1}_{\{h<1\}}(-h\log h) +\mathbf{1}_{\{h\geq 1\}} h\log h\\
    &\leq & e^{-1}+\mathbf{1}_{\{g\geq 1\}}g\log g\\
    &\leq  & e^{-1}+  g\log g +\mathbf{1}_{\{g<1\}}(-g\log g)\\
    &\leq & 2e^{-1} +g\log g
\end{eqnarray*}
as the function $\R_{\geq 0}\ni x\to x\log x$ is non-positive on $[0,1]$ with a minimum at $x=e^{-1}$ while it is positive on $\{x>1\}$.
Set  $h_k:= m^{-1}_\phi e^{\alpha(u_k-u)}\min(g,k) $ and note that $h_k\leq m_\phi^{-1} e^{C}g$ for a uniform constant $C>0$. By the above we have
\begin{eqnarray*}
    \Ent_{\theta}(u_k)&= & \int_X h_k\log (h_k) \,\omega^n \\
    &\leq &  2e^{-1}+ \int_X \log\left(\frac{e^Cg}{m_\phi}\right)\frac{e^Cg}{m_\phi}\omega^n\\
    &=& 2e^{-1}+ e^C\left(\int_X \log(g \,m_\phi^{-1})\frac{g\omega^n}{m_\phi}+ C\int_X \frac{g\omega^n}{m_\phi}\right)\\
    &=&2e^{-1}+ e^C\left(\Ent_\theta(u)+C_1\right)<+\infty.
\end{eqnarray*}
Thanks to Lebesgue Dominated Convergence Theorem we can infer that
$$
\Ent_{\theta}(u_k)=\int_X h_k\log (h_k) \,\omega^n\longrightarrow \Ent_{\theta}(u)= \int_X g\log (g) \,\omega^n.
$$
Very similar arguments show the convergence $\Ent_\eta(v_k)\to \Ent_\eta(v)$. This concludes the proof
\end{proof}

We can now prove Theorem \ref{thm:Convexity_Mabuchi_Anal_Sing}.

\begin{proof}[Proof of Theorem \ref{thm:Convexity_Mabuchi_Anal_Sing}]
Let $(u_t)_{t\in[0,1]}$ be the psh geodesic connecting $u_0,u_1\in \PSH(X,\theta)$, functions with $\phi$-relative minimal singularities.
As $u_t \geq P_\theta(u_0, u_1)$, $u_t$ is with $\phi$-minimal singularities for all $t\in [0,1]$.
Set $v_t=\mathbf{L}(u_t)$, $t\in [0,1]$. 
By construction $v_t$ is an $\eta$-psh function {with minimal singularities} on $Y$. Moreover, Lemma \ref{lem:Bijection}(ii) ensures that $v_t$ is a psh geodesic in $Y$ joining $v_0$ and $v_1$. By \cite[Theorem 4.2]{DNL21} we know that $t\rightarrow \cM_{\eta,w}(v_t) $ is convex in $t$, while \cite[Theorem 3.12]{DDL1} ensures that $t\to E(\eta; v_t,\hat{\varphi})=E(\eta;v_t,0)-E(\eta; \hat{\varphi},0)$ is linear. Thus from Proposition \ref{cor:Mabuchi_UpDown} it follows that
{\small
\begin{multline}\label{eqn:almost convexity}
    \cM_{\theta,\varphi}(u_t)-t\cM_{\theta,\varphi}(u_1)-(1-t)\cM_{\theta,\varphi}(u_0)
    \leq\\
    \frac{n}{\vol(\eta)}\sum_{E_j\not\subset E_{nK}(\eta)} a_j \vol(\eta_{|E_j})\Big(tE_{E_j}(\eta_{|E_j}; v_{1|E_j}, \hat{\varphi}_{|E_j})+(1-t)E_{E_j}(\eta_{|E_j}; v_{0|E_j}, \hat{\varphi}_{|E_j})-E_{E_j}(\eta_{|E_j}; v_{t|E_j}, \hat{\varphi}_{|E_j})\Big).
\end{multline}
}
Set $\mathcal{E}_t:=\Big(tE_{E_j}(\eta_{|E_j}; v_{1|E_j}, \hat{\varphi}_{|E_j})+(1-t)E_{E_j}(\eta_{|E_j}; v_{0|E_j}, \hat{\varphi}_{|E_j})-E_{E_j}(\eta_{|E_j}; v_{t|E_j}, \hat{\varphi}_{|E_j})\Big)$. By the cocycle property of the Monge-Ampère energy (see e.g. Proposition \ref{prop:cocycleE1}), we then have for any $j=1,\dots,m$,
\begin{align*}
    E_{E_j}(\eta_{|E_j}; v_{1|E_j}, \hat{\varphi}_{|E_j})-E_{E_j}(\eta_{|E_j}; v_{t|E_j}, \hat{\varphi}_{|E_j})=E_{E_j}(\eta_{|E_j};v_{1|E_j},v_{t|E_j})\leq \lVert v_1-v_t\rVert_\infty
\end{align*}
 and similarly replacing $v_1$ by $v_0$. Therefore, as $\lVert v_s-v_t\rVert_\infty=\lVert u_s-u_t\rVert_\infty \leq \lvert s-t\rvert \lVert u_0-u_1\rVert_\infty$, we get
\begin{equation}
    \label{eqn:estimate energy on relative canonical}
    \mathcal{E}_t\leq 2t(1-t)\lVert u_0-u_1\rVert_\infty.
\end{equation}
Since $t(1-t)\leq 1/4$, combining (\ref{eqn:almost convexity}) with (\ref{eqn:estimate energy on relative canonical}) yields
$$
\cM_{\theta,\varphi}(u_t)\leq t\cM_{\theta,\varphi}(u_1)+(1-t)\cM_{\theta,\varphi}(u_0)+\frac{n\lVert u_0-u_1\rVert_\infty}{2 \vol(\eta)} \sum_{E_j\not\subset E_{nK}(\eta)} a_j \vol(\eta_{|E_j}).
$$
The proof is finished since $\sum_{j=1}^m a_j [E_j]=[K_{Y/X}]$.
\end{proof}

We conclude this subsection with the following important consequence of Theorem \ref{thm:Convexity_Mabuchi_Anal_Sing}.

\begin{corollary}\label{cor:Bound on entropy along geodesics, gentle analytic singularities case}
    Let $C_1>0$ and let $u_0,u_1\in \PSH(X,\theta)$ with $\phi$-relative minimal singularities such that $\Ent_\theta(u_0),\Ent_\theta(u_1)\leq C_1$. Then there exist positive constants $C_2,C_3$ such that
    $$
    \Ent_\theta(u_t)\leq C_1+C_2+ C_3\{\eta^{n-1}\}\cdot K_{Y/X}
    $$
    for any $t\in [0,1]$. Moreover $C_2,C_3$ only depend on $n,X, \{\omega\},\{\theta\}, \lVert u_0-u_1\rVert_\infty$ and on a lower bound of $m_\phi$.
\end{corollary}
\begin{proof}
    Using \eqref{eqn:Formula almost convexity for analytic singularities} we obtain
    \begin{eqnarray*}
    \Ent_\theta(u_t) &\leq & (1-t)\Ent_\theta(u_0)+t\Ent_\theta(u_1)+\frac{n \lVert u_0-u_1\rVert_\infty}{2m_\phi}\{\eta^{n-1}\}\cdot K_{Y/X}+\\
     &&+\bar{S}_\varphi (1-t)\left({E}(\theta; u_t, \varphi)- {E}(\theta; u_0, \varphi)\right) +\bar{S}_\varphi \, t \left({E}(\theta; u_t, \varphi)-{E}(\theta; u_1, \varphi)\right)\\
    &&+n (1-t)\left({E}_{\Ric(\omega)}(\theta; u_t, \varphi)- {E}_{\Ric(\omega)}(\theta; u_0, \varphi)\right) +nt \left({E}_{\Ric(\omega)}(\theta; u_t, \varphi)-{E}_{\Ric(\omega)}(\theta; u_1, \varphi)\right).
    \end{eqnarray*}
 By the cocycle property \cite[Theorem 5.3]{DDL6}
    \begin{eqnarray*}
       {E} (\theta; u_t, \varphi)-{E}(\theta; u_i, \varphi)&=& \frac{1}{(n+1)\, m_\phi}\sum_{k=0}^{n} \int_X(u_t-u_i) \theta_{u_t}^k\wedge \theta_{u_i}^{n-k}  \\
        &\leq & \frac{\lVert u_t-u_i\rVert_{\infty}}{  m_\phi}\int_X \theta_{\varphi}^{n}
        \leq  \lVert u_t-u_i\rVert_{\infty}.
    \end{eqnarray*}
 
     Again by the cocycle property and the fact that $\Ric(\omega)\leq C \omega$ we get
    \begin{eqnarray*}
     &&  {E}_{\Ric(\omega)} (\theta; u_t, \varphi)-{E}_{\Ric(\omega)}(\theta; u_i, \varphi)\\
       &&= \frac{1}{n\, m_\phi}\sum_{k=0}^{n-1} \int_X(u_t-u_i) \Ric(\omega) \wedge \theta_{u_t}^k\wedge \theta_{u_i}^{n-k-1}  \\
        && \leq  C\frac{\lVert u_t-u_i\rVert_{\infty}}{ m_\phi}\int_X \omega \wedge \theta_{\varphi}^{n-1} \leq C' \lVert u_t-u_i\rVert_{\infty}.
    \end{eqnarray*}
    As $\lVert u_t-u_s \rVert_\infty\leq \lvert t-s\rvert \lVert u_1-u_0 \rVert_\infty$ for any $s,t\in[0,1]$, the conclusion follows.
\end{proof}

\subsection{Transcendental Fujita Approximation}\label{sec:fuj}
We give the following transcendental definition of the well-known Fujita approximation of big line bundle on projective varieties \cite{Fuj94}.
\begin{definition}\label{defi:Transcendental Fujita}
    We say that a sequence of model type envelopes $(\phi_k)_k\subset \PSH(X,\theta)$ is a \emph{transcendental Fujita approximation of $\{\theta\}$} if
    \begin{itemize}
        \item[i)] $\phi_k\in \mN_\theta$ for any $k\in \N$;
        \item[ii)] $\int_X \theta_{\phi_k}^n\to \vol(\theta)$ as $k\to +\infty$.
    \end{itemize}
    We also say that a transcendental Fujita approximation $(\phi_k)_k$ is \emph{monotone} if $\phi_k\nearrow V_\theta$.
\end{definition}
We note that as an immediate consequence of Theorem \ref{thm:weak convergence}, any $(\phi_k)_k$ such that $\phi_k\nearrow V_\theta$ satisfies the condition in (ii).

\medskip

\noindent The following result gives another interpretation of such transcendental Fujita approximation:
\begin{lemma}
    \label{lem:Classic Fujita}
    There exists a transcendental Fujita approximation of $\{\theta\}$ if and only if there exists a sequence of data $(\pi_k,\beta_k,F_k)$, where $\pi_k:Y_k\to X$ is a modification from $Y_k$ compact Kähler manifold od complex dimension $n$, $\beta_k$ is a big and nef class, $F_k$ is an effective $\R$-divisor, such that
    \begin{itemize}
        \item[i)] $\pi_k^*\{\theta\}=\beta_k+\{F_k\}$ for any $k\in \N$;
        \item[ii)] $\vol(\beta_k)\to \vol(\theta)$ as $k\to +\infty$.
    \end{itemize}
\end{lemma}
\begin{proof}
By definition of $\mN_\theta$ if $(\phi_k)_k$ is a transcendental Fujita approximation then there exist modifications $\pi_k:Y_k\to X$, currents with minimal singularities $S_k$ representing big and nef classes $\beta_k$ and effective $\R$-divisors $F_k$ such that $\pi_k^*\theta_{\phi_k}=S_k+[F_k]$ for any $k\in \N$. Thus one implication follows simply observing that $\int_X \theta_{\phi_k}^n=\int_{Y_k} S_k^n=\vol(\beta_k)$.\\
Vice-versa, assume to have a sequence of data $(\pi_k,\beta_k,F_k)_k$ as in the statement. Since $F_k$ is effective, for any $S_k$ current with minimal singularities in $\beta_k$ there exists a unique current $T_k=\theta+dd^c u_k$ such that
$$
\pi_k^*(\theta+dd^c u_k)=S_k+[F_k]
$$
(see \cite[Proposition 1.2.7.(ii)]{BouThesis}). Set $\phi_k:=P_\theta[u_k]$. By \cite[Lemma 5.1]{DDL6} we know that $\phi_k$ and $u_k$ have the same multiplier ideal sheaf and in particular they have the same Lelong numbers on any modification of $X$. Thus, since $S_k$ has minimal singularities and $\phi_k$ is less singular than $u_k$, we infer that
$$
\pi_k^*(\theta+dd^c \phi_k)=\tilde{S}_k+[F_k]
$$
for a positive and closed current with minimal singularities $\tilde{S}_k$ in $\beta_k$, i.e. $\phi_k\in\mN_\theta$. Moreover, as noticed above, $\int_X\theta_{\phi_k}^n=\int_Y S_k^n=\vol(\beta_k)$.
\end{proof}

The existence of a monotone trascendental Fujita approximation is basically a consequence of \cite{Dem92}:
\begin{lemma}
    \label{lem:Fujita}
There exists a monotone transcendental Fujita approximation of $\{\theta\}$.
\end{lemma}
\begin{proof}
By Lemma \ref{lem:properties N_theta}(ii) and the lines below Definition \ref{defi:Transcendental Fujita} it is enough to produce a sequence of $\theta$-psh functions with analytic singularities $\hat{\psi}_k$ such that $\int_X \theta_{\hat{\psi}_k}^n>0$ and such that $\phi_k:=P_\theta[\hat{\psi}_k]$ increases to $V_\theta$.

An immediate consequence of the proof of Demailly's approximation theorem \cite[Proposition  3.7]{Dem92} is that for any $\theta_\psi$ K\"{a}hler current there exists $\psi' \geq \psi$ such that $\theta_{\psi'}$ is a K\"{a}hler current with 
analytic singularities (each element of the approximating sequence satisfies this property).
Moreover observe that if $\theta_{\psi_1}, \theta_{\psi_2}$ are K\"{a}hler currents, then $\theta_{\max(\psi_1,\psi_2)}$ is a K\"{a}hler current as well since by \cite[Lemma 2.9]{DDL6}

$$\theta_{\max(\psi_1,\psi_2)} \geq {\bf 1}_{\{\psi_2\leq \psi_1\}} \theta_{\psi_1} +{\bf 1}_{\{\psi_1< \psi_2\}} \theta_{\psi_2}.$$   

Now let $
\theta_{\psi}$  be a K\"{a}hler current and let $\psi_k := \frac{1}{k} \psi + \left(1 - \frac{1}{k}\right) V_\theta.$  
Then 
$$ \theta_{\psi_k} =  \frac{1}{k} \theta_\psi  + \left(1-\frac{1}{k}\right) \theta_{V_\theta}$$  
is a K\"{a}hler current and  $\int_X \theta_{\psi_k}^n \to \int_X  \theta_{V_\theta}^n $ as $k$ goes to $+ \infty.$

Let $\hat{\psi}_1 $  be a $\theta $-psh function such 
that $\theta_{\hat{\psi}_1}$ is a 
K\"{a}hler current with 
analytic singularities and  
$\hat{\psi}_1 \geq \psi_1.$ 
Inductively let 
$\hat{\psi}_{k+1} $ be a 
 $\theta $-psh function such that $\theta_{\hat{\psi}_{k+1}}$ is a K\"{a}hler current with 
 analytic singularities and 
 $\hat{\psi}_{k+1} \geq     \max(\hat{\psi}_k, \psi_{k+1})$.
Then by construction we have that $\hat{\psi}_k$ is an increasing sequence, 
$ \hat{\psi}_k \geq \psi_k$
and $\theta_{\hat{\psi}_k}$  is a K\"{a}hler current with analytic singularities. 

By \cite[Theorem 1.2]{WN19} we know that $\int_X \theta_{{\psi}_k}^n \leq \int_X \theta_{\hat{\psi}_k}^n \leq \vol(\theta)$. Thus
$\int_X \theta_{\hat{\psi}_k}^n  
\to \int_X  \theta_{V_\theta}^n $ as $k$ goes to $+\infty$. We then consider $\phi_k:=P_\theta[\hat{\psi}_k]$. This sequence is increasing, it has analytic singularity type and $|\phi_k- \hat{\psi}_k|\leq C_k$, for some $C_k>0$.   
    Moreover, by construction, $\sup_X \phi_k=0 $ and $\phi_k\leq V_\theta$. \\
    Let $\phi = (\sup_k \phi_k)^* \in\PSH(X,\theta),$ where $^*$ is the upper semicontinuous regularization. Then $\sup_X \phi =0$
    and by \cite[Remark 2.4]{DDL6}   we have
    $$
    \int_X \theta_{\phi_k}^n= \int_X \theta_{\hat{\psi}_k}^n \to \int_X \theta_{V_\theta}^n
    $$ as $k \to + \infty.$
    Hence, $$  \int_X \theta_{V_\theta}^n = \int_X \theta_\phi^n.$$ 
    
   On the other hand $\phi$ is an increasing limit of model type envelopes, hence  by  \cite[Corollary 4.7]{DDL5} it is a model type envelope,  
    therefore
    $\phi=V_\theta$, concluding the proof.
\end{proof}


\subsection{(Almost) Convexity of the Mabuchi functional}
Let $\{\phi_k\}_{k\in\N}$ be a monotone transcendental Fujita approximation of $\{\theta\}$ (see Definition \ref{defi:Transcendental Fujita}) 
and set $V_k:=\int_X \theta_{\phi_k}^n $, $ V:=\mathrm{Vol}(\theta)$. We note that by definition we have $V_k>0$ for any $k$. 
Let $u   \in \mathcal{E}(X,\theta)$ and 
let us consider  the function $$u_{k}:=P_\theta[\phi_k](u).$$
We wish to collect some properties of the correspondence $u \to u_k.$ 
\begin{lemma} \label{lem_projection}
Let $u \in \mathcal{E}(X,\theta),$ then the correspondence $u \to u_k$ has the following properties:
 \begin{itemize}
        \item[(i)] $u_k \in \mathcal{E}(X,\theta,\phi_k).$
        \smallskip
        \item[(ii)] The sequence $u_{k}$ increases to $u$ as $k$ goes to $+ \infty$ outside of a pluripolar set.
        \smallskip
        \item[(iii)] The set $$S :=  \{ x \in X : u_{k}(x) = u(x)  \ 
        \mbox{for some $k \geq 1$} \} $$  
        $$ =  \{ x \in X :  \mbox{for some $k \geq 1,$ and for all $l \geq k$ we have \ } u_{l}(x) = u(x)  \  \} $$  has full mass with respect to $\theta_{u}^n.$
        \smallskip
        \item[(iv)] If $u,v \in \psh(X,\theta)$  satisfy $|u- v| \leq C$   then  $|u_k-v_k| \leq C$  for all $k \geq 1.$  
        \smallskip
        \item[(v)] We have  $ \theta_{u_k}^n = g_k \theta_u^n$ with $0 \leq g_k \leq 1$ and $g_k$ increasing almost everywhere with respect to $\theta_u^n$ to the constant function $1$ as $k \to + \infty.$ In particular if $\Ent_\theta(u) $  is finite, then $\Ent_\theta(u_k)$  is uniformly bounded independent of $k \geq 1.$ 
        \item[(vi)]  Assume $u_0,u_1\in \mathcal{E}^1(X,\theta) $ have the same singularity type. Let $t\to u_t$ be the  psh geodesic defined in the interval $[0,1]$ joining $u_0,u_1$,
        and let  $t\to u_{t,k}$ be the  psh geodesic joining $u_{0,k},u_{1,k}.$ Then for all $t \in [0,1]$  the sequence $ u_{t,k}$ is increasing to $u_t$ outside of a pluripolar set as $k$ goes to $+ \infty.$ 
        
    \end{itemize} \label{u,u_k}
\end{lemma} 

\begin{proof}
Observe that $u+C\in \mathcal{E}(X, \theta)$ for all $C\in \R$.  Since $\int_X \theta_{\phi_k}^n>0$, $\phi_k=P_\theta[\phi_k]$ and $P_\theta(V_\theta)=V_\theta$, if we apply  \cite[Proposition 5.3]{DDL5}  with $\Phi = \phi_k, $ and $\Psi = V_\theta$ we obtain that  $P_\theta(u+C, \phi_k)\in \mathcal{E}(X, \theta, \phi_k)$. By definition $P_\theta[\phi_k](u) \leq \phi_k$ and it is the increasing limit of $P_\theta(u+C, \phi_k)$, hence
$$\int_X \theta_{\phi_k}^n=\int_X \theta^n_{P_\theta(u+C, \phi_k)} \leq \int_X \theta^n_{P_\theta[\phi_k](u)} \leq \int_X \theta_{\phi_k}^n,$$
this proves (i).\\
Let $\hat{u}$ be the upper semi-continuous regularization of the limit of the increasing sequence $u_{k},$ then $\hat{u} \leq u$ and the set 
$E = \{ x \in X : \sup_k u_k < \hat{u}(x)  \ \mbox{or}  \ \hat{u}(x) = - \infty \}$ is pluripolar, hence it has $\theta_u^n$ measure $0.$ Now $\theta_{ \hat{u}}^n$ is the weak limit of $\theta_{u_{k}}^n$  by Remark \ref{rk:conv_incr}.  
By (i),  $\int_X \theta_{u_k}^n = \int_X \theta_{\phi_k}^n \to \vol(\theta)$ as $k$ goes to $+ \infty.$
Since $\hat{u} \leq u,$ it follows that   
 $\int_X \theta_{\hat{u}}^n = \int_X \theta_{u}^n = \vol(\theta).$ 
By \cite[Theorem 3.8]{DDL2} we have 
$$\theta_{u_{k}}^n\leq \mathbf{1}_{\{u_{k}=u\}}\theta_{u}^n,$$ therefore 
$$\theta_{\hat{u}}^n \leq \mathbf{1}_{S} \theta_{u}^n.$$
Since $u$ and  $\hat{u}$ have the same mass,
$$\theta_{\hat{u}}^n = \theta_{u}^n$$ and $S$ has full mass with respect to $\theta_u^n.$
Therefore $\hat{u}$ and $u$ differ by a constant (see for example \cite[Theorem 3.13]{DDL6}.
Now if $x_0 \in S \setminus E,$ we have $u(x_0) = \hat{u}(x_0) > - \infty,$ hence $u = \hat{u}.$ This proves (ii) and (iii).
The map $u \to P_\theta[\phi_k](u)$  is monotone and  $P_\theta[\phi_k](u +C) = P_\theta[\phi_k](u) + C$ for all $C \in \mathbb{R}.$ 
Then (iv) follows.
Now, the inequality $\theta_{u_{k}}^n\leq \mathbf{1}_{\{u_{k}=u\}}\theta_{u}^n$ also  says that $$\theta_{u_{k}}^n = g_k \theta_{u}^n,$$ with $0 \leq g_k \leq 1.$ Moreover for $j \leq k$  we have $\phi_j \leq \phi_k$ and  $$P[\phi_j](u_k)=P[\phi_j](P[\phi_k](u)) = P[\phi_j](u)= u_j.$$
    \cite[Theorem 3.8]{DDL2} implies that $g_k$ is increasing in $k$ since
    $$ g_j \theta_u^n=\theta_{u_j}^n\leq \mathbf{1}_{\{u_{j}=u_k\}}\theta_{u_k}^n \leq \theta_{u_k}^n = g_{k} \theta_{u}^n.$$ Now by the above $\theta_{u_k}^n $ converges weakly to $\theta_u^n,$ hence $g_k$ converges to the constant function $1$ almost everywhere with respect to $\theta_u^n.$
Moreover since  $\log(m_{\phi_k})$ is uniformly bounded we derive (v).

In order to prove (vi) we  note that $ u_{i,k-1} \leq u_{i,k}$, $i=0,1$, hence $u_{t,k-1}$ is a subgeodesic with respect to the end points $u_{0,k}, u_{1,k}$. This means that the sequence $u_{t,k}$ is increasing in $k.$ Moreover by the $t$-convexity of the geodesic and (iv) we have $u_{t,k} \leq \phi_k +C \leq V_\theta + C $ for some positive constant $C$ (indipendent of $k$), hence $u_{t,k} $   increases to a psh subgeodesic segment  $t\to v_t$ that by (ii) is joining $u_0$ and $u_1$. 
Now by maximality $u_{t} \geq P_\theta(u_0,u_1),$ then each $u_t$ has the same singularity type of $u_0$ and $u_1$.
We claim that $v_t=u_t$. \\
  Combining Lemma \ref{lem:Bijection}(iv), Propositions \ref{prop: energy linear}  and \ref{prop: cocy model} we find that 
    $$t\to {E}(\theta; u_{t,k}, \phi_k)$$ is linear.
    Again by the cocycle property (Proposition \ref{prop: cocy model}) we have ${E}(\theta; u_{t,k}, \phi_k)= {E}(\theta; u_{t,k}, u_{0,k})+{E}(\theta; u_{0,k}, \phi_k).$ We then infer that $t\to {E}(\theta; u_{t,k}, u_{0,k})$ is linear.
Moreover, since $u_{t,k}$ is increasing to $u_t$, $u_{0,k}$ is increasing to $u_0$ and by (iv) we know that $(u_{t,k}-u_{0,k})$ is uniformly bounded, thanks to Theorem \ref{thm:weak convergence} we ensure that
    $$
    {E}(\theta; u_{t,k}, u_{0,k}) =\frac{1}{n+1}\sum_{j=0}^n \int_X (u_{t,k}-u_{0,k})\theta_{u_{t,k}}^j\wedge \theta_{u_{0,k}}^{n-j}\longrightarrow \frac{1}{n+1}\sum_{j=0}^n \int_X (v_t-u_0)\theta_{v_t}^j\wedge \theta_{u_0}^{n-j}={E}(\theta; v_t, u_0)
    $$
    as $k\rightarrow +\infty$. 
Thus $t\to {E}(\theta; v_t, u_0)$ is linear. Using the cocycle property \eqref{prop:cocycleE1} we deduce that $t\to {E}(\theta; v_t, V_\theta)$ is linear. On the other hand $t\to {E}(\theta; u_t, V_\theta)$ is linear as well thanks to Proposition \ref{prop: energy linear}. Thus, for any $t\in[0,1]$ we have
$$ {E}(\theta; u_t, V_\theta)= (1-t)E(\theta; u_0, V_\theta)+tE(\theta; u_1, V_\theta)=  {E}(\theta; v_t, V_\theta). $$
By maximality of the psh geodesic, we have $u_t\geq v_t$. Hence $u_t=v_t$ by \cite[Lemma 2.9]{DDL5}. This proves (vi).
\end{proof}

\medskip

Now we want to apply the results of the previous subsection \ref{sec:Analytic_Singularities} to any monotone transcendental Fujita approximation $(\phi_k)_k$.
For each $k$ we set $\pi_k:Y_k\to X$ a modification from $Y_k$ compact Kähler manifold (of complex dimension $n$) such that
$$\pi_k^*\theta_{{\phi}_k}=(\eta_{k} +dd^c{\tilde{\phi}_k}) + [F_k]$$
where for each $k$, $\eta_k$ is a smooth and closed form representing a big and nef class while $\tilde{\phi}_k$ is a $\eta_k$-psh function with minimal singularities normalized by $\sup_{Y_k}\tilde{\phi}_k=0$.\\
Let $E_{1,k},\dots,E_{m_k,k}$ be the exceptional divisors of $\pi_k$, and $a_{j,k}>0$ such that $K_{Y_k/X}=\sum_{j=1}^{m_k} a_{j,k} E_{j,k}$. 
\medskip

\begin{defi} Given a monotone transcendental Fujita approximation $(\phi_k)_k$ we define

    $$H(\phi_k):= \liminf_{k\to +\infty} \{\eta_k^{n-1}\}\cdot K_{Y_k/X}.$$
    and
    $$H:= \inf\{ H(\phi_k), \, (\phi_k)_k \, {\rm\, monotone\; transcendental\;  Fujita \;approximation}\}.$$  
\end{defi}

\noindent We start by proving that these quantities are well defined.
\begin{lemma}\label{lem:Independence on modification}
  $H(\phi_k)$ only depend on the choice of the transcendental Fujita appro\-ximation $(\phi_k)_k$. 
\end{lemma}

\begin{proof}
Let $\phi\in\mN_\theta$ and let $\pi_1: Y_1\to X$, $\pi_2:Y_2\to X$ be two modifications where $Y_1,Y_2$ are compact Kähler manifolds such that for each $i=1,2$
    $$
    \pi_i^* \theta_{\phi}=\eta_{i,\tilde{\phi}_i}+[F_i]
    $$
    where $F_i$ is an effective $\R$-divisor, $\eta_i$ is a closed and smooth form representing a big and nef class, and $\tilde{\phi}_i$ is a $\eta_i$-psh function with minimal singularities normalized by $\sup_{Y_i}\tilde{\phi}_i=0$. The goal is to prove that
    $$
    \{\eta_1^{n-1}\}\cdot K_{Y_1/X}=\{\eta_2^{n-1}\}\cdot K_{Y_2/X}.
    $$
    Resolving the graph of the bimeromorphic map $\pi_2^{-1}\circ\pi_1: Y_1\dashrightarrow Y_2 $, yields modifications $\rho_1:Z\to Y_1, \rho_2: Z\to Y_2$ such that the diagram
    $$
    \begin{tikzcd}
    & Z \arrow[dr,"\rho_2"] \arrow[dl,"\rho_1",swap] & \\
    Y_1\arrow[dr,"\pi_1",swap] & & Y_2 \arrow[dl,"\pi_2"]\\
    & X &
    \end{tikzcd}
    $$
is commutative. In particular
    \begin{equation}\label{eqn:For Siu}
        \rho_1^* \eta_{1,\tilde{\phi}_1}+[\rho_1^*F_1]=(\pi_1\circ\rho_1)^* \theta_{\phi}=(\pi_2\circ\rho_2)^* \theta_{\phi}= \rho_2^* \eta_{2,\tilde{\phi}_2}+[\rho_2^*F_2].
    \end{equation}
    It follows from \cite[Propositions 3.2, 3.6]{Bou04} that the positive and closed currents $\rho_i^*\eta_{i,\tilde{\phi}_i}$ have zero Lelong numbers everywhere as they are currents with minimal singularities in the big and nef class $\rho_i^*\{\eta_i\}$. Thus, \eqref{eqn:For Siu} gives two decompositions of the same positive and closed current $T:=(\pi_1\circ\rho_1)^* \theta_{\phi}=(\pi_2\circ\rho_2)^* \theta_{\phi}$ into a sum of a current with zero Lelong number and of a current of integration along a divisor. By uniqueness of the Siu's Decomposition \cite{Siu74} (see also \cite[\S 8.(8.16)]{DemAGbook}) we infer that $\rho_1^*\eta_{1,\tilde{\phi}_1}=\rho_2^*\eta_{2,\tilde{\phi}_2}$.
    In particular $\{\rho_1^*\eta_1\}=\{\rho_2^*\eta_2\}$.\\
    Moreover, by   formula \ref{relcan} applied to $\rho_i,$ $\pi_i$  and their compositions, we have $K_{Z/X}=K_{Z/Y_i}+\rho_i^*K_{Y_i/X}$ and $K_{Z/Y_i}$ is $\rho_i$-exceptional.    
    It then 
    follows that
    \begin{eqnarray*}
 \{\eta_1\}^{n-1} \cdot K_{Y_1/X} &= & (\rho_1^*\{\eta_1\})^{n-1} \cdot \rho_1^* K_{Y_1/X} \\
 &=& \{ \rho_1^*\eta_1\}^{n-1} \cdot \left( K_{Z/X} - K_{Z/Y_1}\right)\\
 &=& \{\rho_1^*\eta_1\}^{n-1}\cdot  K_{Z/X},
\end{eqnarray*}
since $K_{Z/Y_1}$ is $\rho_1$-exceptional. \\
An analogous formula holds for $\rho_2, \eta_2, Y_2.$
Since $\{\rho_1^*\eta_1\}=\{\rho_2^*\eta_2\}$, we are done.
\end{proof}


\medskip

\noindent In the following we will work under one of these two conditions:
\begin{equation}\label{UK}\tag{{\bf Condition A}}
H<+\infty
\end{equation}

\begin{equation}\label{YTD}\tag{{\bf Condition B}}
H=0.
\end{equation}
We refer to subsection \ref{sec:YTD} for some examples when these conditions hold and for a digression on how they are related to the uniform version of Yau-Tian-Donaldson conjecture in the algebraic case. 
\smallskip

Our main theorem states as follows:

\begin{theorem}
    \label{thm:Convexity_Mabuchi}
  Let $u_0,u_1\in\psh(X,\theta)$ with minimal singularities and let $(u_t)_{t\in [0,1]}$ be the psh geodesic connecting $u_0$ and $u_1$. Let also $\varphi\in \cE(X,\theta)$ be such that $\theta_\varphi^n=\vol(\theta)\omega^n$, $\sup_X \varphi=0$. Then $u_t$ has minimal singularities and the function  $t\mapsto \cM_{\theta, \varphi}(u_t)$ is almost convex in $[0,1]$, i.e.
\begin{equation}\label{eqn:Almost_Convexity_Statement_Thm}
    \cM_{\theta,\varphi}(u_t)\leq (1-t)\cM_{\theta,\varphi}(u_0)+t\cM_{\theta,\varphi}(u_1)+H \lVert u_0- u_1\rVert_\infty.
\end{equation}
     In particular, if \eqref{YTD} holds, then $\cM_{\theta, \varphi}$ is convex along $u_t$.
\end{theorem}
\begin{proof}
    We can assume that $\Ent_\theta(u_i,\varphi) $ is finite. Indeed otherwise either $\cM_{\theta,\varphi}(u_0)$ or $\cM_{\theta,\varphi}(u_1)$ would be equal to $+\infty$ and the requested inequality would be trivial. Without loss of generality, we can consider $(\phi_k)_k$ be a monotone transcendental Fujita approximation with $H(\phi_k)<+\infty$. Consider $$u_{0,k}:=P_\theta[\phi_k](u_0),  \qquad u_{1,k}:=P_\theta[\phi_k](u_1).$$
By Lemma \ref{u,u_k},  $ u_{i,k}$ has $\phi_k$-relative minimal singularities with a constant independent of $k$ and $u_{i,k}$ converges to $u_i$ for $i=0,1.$

\noindent For $k\in \N$, let $\varphi_k$ be the unique solution in $\mathcal{E}(X,\theta,\phi_k)$ of
    $$
(\theta+dd^c\varphi_k)^n=\Big(\int_X\theta_{\phi_k}^n\Big)\omega^n,\,\,\, \sup_X \varphi_k=0.
    $$
 Thanks to the stability result in \cite[Theorem 1.4]{DDL5} (that can be applied thanks to \cite[Lemma 4.1]{DDL5}) we have that $\varphi_k$ converges in capacity to $\varphi$. \\
 Also, we claim that $|\varphi_k-\phi_k|\leq C$, for $C>0$ independent of $k$. Indeed, by \cite[Theorem 4.7]{DDL4} $|\varphi_k-\phi_k|\leq C_k$, where $C_k=\frac{A(\theta, \omega, n)}{V_k^2} $. Since $V_k$ is increasing we get that $C_k\leq C_1$. The claim is then proved.\\
  Thanks to Theorem \ref{thm:Convexity_Mabuchi_Anal_Sing}  we have the almost convexity of the Mabuchi functional $\cM_{\theta, \varphi_k}$ along psh geodesic segments joining functions with $\phi_k$-relative minimal singularities:
  \begin{eqnarray*}
\cM_{\theta,\varphi_k}(u_t) &\leq & t\cM_{\theta,\varphi_k}(u_{1,k})+(1-t)\cM_{\theta,\varphi_k}(u_{0,k}) +\frac{n\lVert u_{0,k}-u_{1,k}\rVert_\infty}{2 V_k} \{\eta_k^{n-1}\}\cdot K_{Y_k/X}\\
&\leq & t\cM_{\theta,\varphi_k}(u_{1,k})+(1-t)\cM_{\theta,\varphi_k}(u_{0,k}) +\frac{n\lVert u_{0}-u_{1}\rVert_\infty}{2 V_k}\{\eta_k^{n-1}\}\cdot K_{Y_k/X} 
\end{eqnarray*}
where the last inequality follows from $\lVert u_{0,k}-u_{1,k}\rVert_\infty\leq \lVert u_0-u_1\rVert_\infty$.\\ 
By assumption we know that $\theta_{u_i}^n= f_i \theta_\varphi^n= f_i V \omega^n $. 
    By Lemma \ref{u,u_k}(v) $\theta_{u_{i,k}}^n$ has finite entropy w.r.t. $\theta_{\varphi_k}^n$. Let $t\to u_{t,k}$ denote the psh geodesic segment joining $u_{0,k},u_{1,k}$. Then Corollary \ref{cor:Bound on entropy along geodesics, gentle analytic singularities case} together with $H(\phi_k)<+\infty$ and Lemma \ref{lem_projection}(v) ensure that $\theta_{u_{t,k}}^n$ has finite entropy as well for any $t\in [0,1]$ and 
    $$\theta_{u_{t,k}}^n= V_k f_{t,k} \, \omega^n = f_{t,k}  \theta_{\varphi_k}^n, \quad \  \quad \int_X  f_{t,k} \log  f_{t,k} \, \omega^n \leq C,$$ 
    for some $C$ independent of $t$ and $k$.\\
 As mentioned above, $u_{i,k}$ has $\phi_k$-relative minimal singularities with uniform constants, thus by the Lipschitz property of psh geodesics we have that
\begin{equation}
        \label{eq: mingeo}   
\phi_k\geq u_{t,k}\geq \phi_k -C
\end{equation}
   with $C>0$ independent of $k$ and of $t$.
   \smallskip

We now claim that to get \eqref{eqn:Almost_Convexity_Statement_Thm} it is enough to show that
    \begin{equation}
        \label{eqn:Cont_Fujita}
        \cM_{\theta, \varphi_k}(u_{i,k})\longrightarrow \cM_{\theta,\varphi}(u_i),\; i=0,1 
        \end{equation}
        and
       \begin{equation}     
        \label{eqn:Lower_Semic_Fujita}
        \liminf_{k\to +\infty} \cM_{\theta,\varphi_k}(u_{t,k})\geq \cM_{\theta, \varphi}(u_t), \; \forall t\in (0,1).
 \end{equation}
Indeed, \eqref{eqn:Cont_Fujita} and \eqref{eqn:Lower_Semic_Fujita}, together with the almost convexity of $\cM_{\theta, \varphi_k}$, imply
   \begin{eqnarray*}
     \cM_{\theta,\varphi}(u_{t}) &\leq & (1-t) \cM_{\theta,\varphi}(u_{0}) + t \cM_{\theta,\varphi}(u_{1}) +\frac{n\lVert u_0-u_1\rVert_\infty }{2V} \liminf_{k\to +\infty} \{\eta_k^{n-1}\}\cdot K_{Y_k/X}\\
      &= & (1-t) \cM_{\theta,\varphi}(u_{0}) + t \cM_{\theta,\varphi}(u_{1}) + \frac{n\lVert u_0-u_1\rVert_\infty }{2V} H(\phi_k).
   \end{eqnarray*}
Taking the infimum over all monotone transcendental Fujita approximations we conclude.
\smallskip

We now prove \eqref{eqn:Cont_Fujita} and \eqref{eqn:Lower_Semic_Fujita}.

\smallskip
\medskip

    \textbf{Step 1: Convergence of the energies.} Since $u_{t,k}- \varphi_k$ is uniformly bounded, $u_{t,k} \nearrow u_t$, $\varphi_k \to\varphi$ in capacity, $u_{t,k}$ and $\varphi_k$ are more singular than $u_{t}$ and $\varphi$, respectively Remark \ref{rk:conv_incr} and Theorem \ref{thm:weak convergence}
    ensure that for any $j=0, \cdots,n $, we have
    
$$ \int_X (u_{t,k}-\varphi_k) \theta_{u_{t,k}}^j\wedge \theta_{\varphi_k}^{n-j}\longrightarrow \int_X (u_t-\varphi) \theta_{u_t}^j\wedge \theta_{\varphi}^{n-j} $$
and 
$$ \int_X (u_{t,k}-\varphi_k) \theta_{u_{t,k}}^j\wedge \theta_{\varphi_k}^{n-j-1}\wedge \Ric(\omega)\longrightarrow  \int_X (u_t-\varphi) \theta_{u_t}^j\wedge \theta_{\varphi}^{n-j-1}\wedge \Ric(\omega).$$
We then deduce that ${E} (\theta; u_{t,k}, \varphi_k)\to {E}(\theta; u_t,\varphi)$ and ${E}_{\Ric(\omega)}(\theta; u_{t,k},\varphi_k)\to {E}_{\Ric(\omega)}(\theta; u_t, \varphi)$, for any $t\in [0,1]$,
\smallskip

    \textbf{Step 2: Lower semicontinuity of the entropy.}
   As $V_k^{-1}\theta_{\varphi_k}^n=\omega^n$, it follows from \cite[Proposition 2.10]{BBEGZ19} that for any $\theta$-psh function $v$ we have
    $$
    \Ent_\theta(v,\varphi_k)=\Ent( m_v^{-1} \theta_v^n, V_k^{-1}\theta_{\varphi_k}^n)=\Ent( m_v^{-1} \theta_v^n, \omega^n)=\sup_{g\in C^0(X)} \Big(\int_X g\, \frac{\theta_v^n}{m_v}-\log \int_X e^g \omega^n\Big).
    $$
    In particular the functional $\Ent_\theta(\cdot,\varphi_k)=\Ent( \cdot, \omega^n)$ is lower semicontinuous on the space of pro\-bability measures on $X$ with respect to the weak convergence. Since $u_{t,k}\nearrow  u_t$, $ \theta_{u_{t,k}}^n$ converges weakly to $\theta_{u_t}^n$ and 
    $$
    \liminf_{k\to +\infty}\Ent_\theta(u_{t,k},\varphi_k)=  \liminf_{k\to +\infty}\Ent(V_k^{-1}\theta^n_{u_{t,k}},\omega^n)\geq \Ent_\theta(V^{-1}\theta^n_{u_t}, \omega^n)= \Ent_\theta(u_t,\varphi).
    $$
  
    Next, for $i=0,1$,  we write $\theta_{u_{i,k}}^n=g_{i,k} \theta^n_{u_i}$ where $0 \leq  g_{i,k}\leq 1$ and $g_{i,k} \nearrow 1$ almost everywhere with respect to $\theta_{u_i}^n$ by Lemma \ref{u,u_k}.
    Since 
    $$
    V_k=\int_X \theta_{u_{i,k}}^n=\int_X g_{i,k}\theta_{u_i}^n\longrightarrow V=\int_X \theta_{u_i}^n
    $$
    and $\theta_{u_i}^n=Vf_i\omega^n$ we obtain that $\| f_i(1-g_{i,j})\|_{L^1(\omega^n)} \rightarrow 0$ as $j\rightarrow +\infty$. By Lemma \ref{u,u_k} and by the dominated convergence theorem we deduce that
    $$
    \Ent_\theta(u_{i,k}, \varphi_k )=\frac{V}{V_k^2} \int_X g_{i,k} f_i \log \left( \frac{V}{V_k} g_{i,k} f_i \right) \theta_{\varphi_k}^n=\frac{V}{V_k}  \int_X g_{i,k} f_i \log \left( \frac{V}{V_k} g_{i,k} f_i \right) \omega^n
    $$
    converges as $k\to +\infty$ to
    $$
    \int_X f_i \log f_i\, \omega^n = \frac{1}{V}\int_X f_i \log f_i\, \theta_\varphi^n= \Ent_\theta(u_i,\varphi).
    $$
    \smallskip

    \textbf{Step 3: Conclusion of the proof.}
    Since $\varphi_k\simeq \phi_k$ and $\varphi\simeq V_\theta$, after re-writing $\Ric(\omega)=T_2-T_1$ for some smooth K\"ahler forms, \cite[Proposition 2.1]{DDL2} ensures that $$\bar{S}_{\varphi_k}=\frac{n}{V_k}\int_X \Ric(\omega)\wedge \theta_{\varphi_k}^{n-1}=\frac{n}{V_k}\int_X \Ric(\omega)\wedge \theta_{\phi_k}^{n-1}, \quad \bar{S}_{\varphi}=\frac{n}{V}\int_X \Ric(\omega) \wedge \theta_{\varphi}^{n-1}=\frac{n}{V}\int_X \Ric(\omega) \wedge \theta_{V_\theta}^{n-1}.$$
    Also, since $\phi_k$ increases to $V_\theta$, Theorem \ref{thm:weak convergence} gives that $T_i \wedge \theta_{\phi_k}^{n-1}$ weakly converges to $T_i \wedge \theta_{V_\theta}^{n-1}$ for $i=1,2$. As $V_k\nearrow V$, we have that $\bar{S}_{\varphi_k}$ converges to $\bar{S}_{\varphi}$.
    Then Step 2 and 3 give the required convergences \eqref{eqn:Cont_Fujita}, \eqref{eqn:Lower_Semic_Fujita}.
    \end{proof}

As a corollary we obtain:

\begin{corollary}\label{cor: bound entropy geo}
   Let $C_1>0$ and let $u_0,u_1\in \PSH(X,\theta)$ such that $|u_0-u_1|\leq C$, $C>0$. Assume $\Ent_\theta(u_0),\Ent_\theta(u_1)\leq C_1$. Then there exists positive constants $C_2, C_3$ such that
    $$
    \Ent_\theta(u_t)\leq C_1+C_2+C_3H
    $$
    for any $t\in [0,1]$. Moreover $C_2, C_3$ only depends on $n,X, \{\omega\},\{\theta\}, \lVert u_0-u_1\rVert_\infty$, and on a lower bound of $\vol(\theta)$.
\end{corollary}
\begin{proof}
  When $u_0, u_1$ have minimal singularities, the same arguments in the proof of Corollary \ref{cor:Bound on entropy along geodesics, gentle analytic singularities case} give the conclusion. In the general case we can adapt the arguments in \cite[Proposition 4.3]{DNL21}. We observe indeed that in the proof \cite[Proposition 4.3]{DNL21} the authors use the distance $d_1$ on the space $\mathcal{E}^1(X,\theta)$ when $\theta$ is big and nef. However this distance has been defined in the big case as well and all the relevant properties have been proved \cite{DDL3}. Another key ingredient in the proof of \cite[Proposition 4.3]{DNL21} is the convexity of the Mabuchi functional, but the almost convexity in our case (Theorem \ref{thm:Convexity_Mabuchi}) suffices.
\end{proof}

\subsection{On the Condition B and the Yau-Tian-Donaldson Conjecture}\label{sec:YTD}

In this subsection we prove that if $V_\theta\in \mN_\theta$ then \eqref{YTD} is satisfied. In particular, thanks to Theorem \ref{thm:Convexity_Mabuchi} the Mabuchi functional $\cM_{\theta,\varphi}$ is convex along geodesic segments joining potentials with minimal singularities.

\medskip

We recall that, by \cite{Bou04}, any pseudoeffective cohomology class $\alpha\in H^{1,1}(X,\R)$ admits a \emph{divisorial Zariski decomposition}. When the class $\alpha$ is big such decomposition can be described as the Siu Decomposition of a positive and closed current with minimal singularities.
\begin{prop}\label{prop:Zariski}
    Let $\alpha\in H^{1,1}(X,\R)$ be a big class, $\theta$ be a smooth closed $(1,1)$-form in $\alpha$ and let $T_{\min}:=\theta_{V_\theta}$. Then the divisorial Zariski decomposition of $\alpha$ is given as
    $$
    \alpha=\mathcal{P}(\alpha)+\mathcal{N}(\alpha)
    $$
    where the negative and positive parts $\mathcal{N}(\alpha), \mathcal{P}(\alpha)$ are given as follows:
    \begin{itemize}
        \item[i)] The negative part $\mathcal{N}(\alpha)$ is the cohomology class of $\sum_D \nu(T_{\min}, D)[D]$ where the sum is over all prime divisors on $X$ and where there is only a finite number of prime divisors $D$ for which $\nu(T_{\min},D) > 0$.
        \item[ii)] The positive part $\mathcal{P}(\alpha)$ is defined as the difference $\alpha-\mathcal{N}(\alpha)$, and the current $T_{\min} - \sum_D \nu(T_{\min}, D)[D]$ has minimal singularities in the class $\mathcal{P}(\alpha).$ In particular  $\vol(\mathcal{P}(\alpha)) = \vol(\alpha).$ 
        \end{itemize}
\end{prop}
\begin{proof}
    (i) follows from \cite[Proposition 3.6, Definition 3.7, Theorem 3.12]{Bou04}. To prove (ii) we observe that if $S_{\min}$ is a positive current with minimal singularities in $\mathcal{P}(\alpha)$, then $S_{\min} + \sum_D \nu(T_{\min}, D)[D]$ is more singular than $T_{\min}$. Thus $T_{\min} - \sum_D \nu(T_{\min}, D)[D]$ is a positive and closed current less singular than $S_{\min}$, so they have the same singularities. 
\end{proof}

Since $\alpha=\{T_{\min}\}$, as direct consequence of Siu's Decomposition 
we have that $\mathcal{P}(\alpha)$ is the cohomology class of a positive and closed current with zero Lelong numbers along any prime divisor on $X$. The class $\mathcal{P}(\alpha)$ is said to be \emph{nef in codimension 1} or also \emph{modified nef}.

\begin{definition}\footnote{Such definition can be given for pseudoeffective classes but for the purposes of the paper we only consider big classes.}
    Let $\alpha$ be a big cohomology class.
    We say that $\alpha$ admits \emph{a Zariski decomposition} if $\mathcal{P}(\alpha)$ is nef.
    Moreover we say that $\alpha$ admits \emph{a bimeromorphic Zariski decomposition} if there exists a modification $\mu:Y \to X$ from $Y$ compact Kähler manifold of complex dimension $n$, such that $\mu^*\alpha$ admits a Zariski decomposition.
\end{definition}

    If $\dim X=2$ then it is well-known that any big class (actually any pseudoeffective class) admits a Zariski decomposition, generalizing the pioneering work of Zariski \cite{Zar62}. 
    However there are example of big classes that do not admit a bimeromorphic Zariski decomposition: see for instance \cite[Section A.2]{Bou04}.
\smallskip

\noindent We are now ready to observe the following:

\begin{prop}\label{prop:V_theta in N_theta}
   $V_\theta\in \mN_\theta$ if and only if $\{\theta\}$ admits a bimeromorphic Zariski decomposition.
\end{prop}
\begin{proof}
   Set $\alpha=\{\theta\}.$   Assume that $\mu:Y\to X$ is a modification from $Y$ compact Kähler manifold such that $\mu^*\alpha$ admits a Zariski decomposition $\mu^*\alpha=\mathcal{P}(\alpha)+\mathcal{N}(\alpha)$. Again from \cite{BouThesis} we know that any positive closed real $(1,1)$ current in $\mu^*(\alpha)$ 
   is the pull-back of a positive closed real  $(1,1)$ current in $\alpha,$ so $\mu^*(\theta_{V_\theta})$ is a current with minimal singularities in $\mu^*(\alpha).$
   From Proposition \ref{prop:Zariski} we know that
    \begin{equation}
        \label{eqn:V_theta in N_theta}
        \mu^*\theta_{V_\theta}=S+[F]
    \end{equation}
    where $S$ is a current with minimal singularities in $\mathcal{P}(\alpha)$ while $F$ is an effective $\R$-divisor in $\mathcal{N}(\alpha)$. 
    Since by assumption $\mathcal{P}(\alpha)$ is nef we conclude that $V_\theta\in\mN_\theta$.\newline
    Vice-versa suppose that $V_\theta\in \mN_\theta$, i.e. there exist a modification $\mu:Y\to X$ from $Y$ compact Kähler manifold, a current with minimal singularities $S$ representing a big and nef class and an effective $\R$-divisor $F$ such that \eqref{eqn:V_theta in N_theta} holds. Then since, $\mu^*(\theta_{V_\theta})$ is with minimal singularities and  the non-pluripolar product does not charge pluripolar sets, we obtain that
    $$
    \vol(\mu^*\alpha)=\int_Y \mu^*\theta_{V_\theta}=\int_Y S^n=\vol(S).
    $$
    Hence it follows from \cite[Main Theorem]{DNFT17} that $\{S\}$ is the positive part in the Zariski decomposition of $\mu^*\alpha$, i.e. $\mu^*\alpha$ admits a Zariski decompositon as $\{S\}$ is nef.
\end{proof}

We are now ready to prove the main result of this subsection.
\begin{thm}\label{thm:We_have_Convexity}
    Let $\alpha=\{\theta\}$ be a big cohomology class that admits a bimeromorphic Zariski decomposition. Assume $\varphi\in \mathcal{E}(X,\theta)$ such that $\theta_\varphi^n=\vol(\theta) \omega^n$. Then $\cM_{\theta,\varphi}$ is convex along psh geodesics joining potentials with minimal singularities.
\end{thm}
It  then follows that $\cM_{\theta,\varphi}$ is convex when $dim_\C X=2$.
\begin{proof}
    By Proposition \ref{prop:V_theta in N_theta} we know that $V_\theta\in \mN_\theta$. Let $\mu:Y\to X$ be a modification from $Y$ compact Kähler manifold such that 
    $$
    \mu^*\theta_{V_\theta}=\eta_{\tilde{\phi}}+[F]
    $$
    where $\eta_{\tilde{\phi}}$ is a positive and closed current with minimal singularities representing a big and nef class and $F$ is an effective $\R$-divisor. By Theorem \ref{thm:Convexity_Mabuchi_Anal_Sing} it is enough to check that
    $$
    \{\eta^{n-1}\}\cdot K_{Y/X}=0.
    $$
However as observed in the proof of Proposition \ref{prop:V_theta in N_theta}, $\{\eta\}=\mathcal{P}(\mu^*\alpha)$ and by \cite[Lemmas 4.3, 6.1]{CT16} it follows that $ \rmE_{nK}(\mu^*\alpha)=\rmE_{nK}(\mathcal{P}(\mu^*\alpha)).$  Thus by the main result in \cite{CT15}, $\vol(\eta_{|E})=0$ for any exceptional divisor $E$. As $K_{Y/X}=\sum_{j=1}^ma_j E_j$, we infer that $\{\eta^{n-1}\}\cdot K_{Y/X}=0$, which concludes the proof.
\end{proof}

\begin{remark}\label{rk:B YTD} 
  We claim that if $X$ is a projective manifold, $\alpha$ is the cohomology class of a big $\Q$-divisor and there exists a birational morphism $\pi:Y\to X$ such that $\pi^*\alpha$ admits a Zariski decomposition then it is possible to produce a Fujita approximation in the sense of \cite{Fuj94}, and such that $\{\eta_k^{n-1}\}\cdot K_{Y_k/X}\to 0$. \\
  Indeed, as explained in \cite[Lemma 4.13]{Li23}, we can perturb the class $\{\eta\}=\mathcal{P}(\pi^*\alpha)$ and construct a sequence of ample classes $\{\eta_k\}_{k\in \N}$ on $Y_k=Y$ such that $\pi^*\alpha-\{\eta_k\}=\{D_k\}$ for a decreasing sequence of effective $\Q$-divisors $D_k$ and such that $\{\eta_k\}_k$ converges to $\mathcal{P}(\pi^*\alpha)$. \\
In this situation, the proof of Lemma \ref{lem:Classic Fujita} then ensures that we have a sequence of model potentials $\phi_k\in \mN_\theta$. One can check that $\phi_k$ is an increasing sequence. 
Moreover, its limit is a model potential as well thanks to \cite[Corollary 4.7]{DDL5} and $\int_X \theta_{\phi_k}^n=\vol_Y(\eta_k)\to \vol_X(\theta)=\int_X \theta_{V_\theta}^n$. We can then infer that $\phi_k \nearrow V_\theta$.\\
We thus get $\{\eta_k^{n-1}\}\cdot K_{Y/X}\to \{\eta^{n-1}\}\cdot K_{Y/X}=0$ where the latter equality follows as in the proof of Theorem \ref{thm:We_have_Convexity}. As observed in \cite[Lemma 4.5]{Li23} (and before in \cite[Conjecture 2.5]{BJ18}), producing a Fujita approximation (for the so-called \emph{big models}) such that \eqref{YTD} holds for semiample classes $\{\eta_k\}$ implies a resolution to the Yau-Tian-Donaldson Conjecture.\\
  In the general case of big classes on compact Kähler manifolds, clearly \eqref{YTD} can be seen as a trascendental extension of \cite[Conjecture 4.7]{Li23}.
\end{remark}

\section{Monge-Amp\`ere measures on contact sets}

The following result extends \cite[Theorem 1.2]{DNL21} to big classes.

\begin{theorem}
\label{thm:Geodesic_Formula2}
Assume \eqref{UK}. Let $u_0,u_1\in \mathcal{E}(X,\theta)$ such that $u_0,u_1\in \Ent(X,\theta).$ Let $u_t$ be the psh geodesic connecting $u_0$ and $u_1$. Fix $p\geq 1$. If $u_0-u_1$ is bounded then  $\dot{u}_t^+,$ and $\dot{u}_t^-$  and $|u_t-u_0|$ are  uniformly bounded, and $\dot{u}_t^+=\dot{u}_t^- := \dot{u}_t$ almost everywhere with respect to $\theta^n_{u_t}.$ Moreover 
$$
\int_X | \dot{u}_t|^p \,\theta_{u_t}^n
$$
is independent of  $0 \leq t \leq 1$. 
\end{theorem}

 \begin{proof}
 First  of all recall that each geodesic is convex in $t,$ moreover as $|u_0 - u_1| $ is uniformly bounded by a positive constant $C,$ so is each derivative $\dot{u}^+_t$ and $\dot{u}^-_t$ as well as each incremental ratio $\frac{u_{t+h} -u_t}{h}$ and $\frac{u_{t-h} -u_t}{-h}$ with $h > 0$ (see \eqref{eq: Lip}). In particular 
$|u_t- u_0| \leq Ct \leq C.$

By Lemma \ref{u,u_k} (items (iv) and (v)) we know that
$|u_{1,k} - u_{0,k}|  \leq \sup_X |u_1 - u_0| \leq C,$ and $u_{0,k},u_{1,k}\in \Ent(X,\theta)$ for any $k\in\N$. Let $u_{t,k}$ be the psh geodesic connecting $u_{0,k}$ and $u_{1,k}$. 
  It then follows that 
$|\dot{u_{k,t}}^+|, | \dot{u_{k,t}}^-|,|\dot{u_t}^+|, | \dot{u_t}^-| $ 
are uniformly bounded independently of $k \in \N, t  \in [0,1],$ and $x \in X.$

   We now let $\phi_k$ be a monotone transcendental Fujita approximation of $\{\theta\}$ (whose existence is ensured by Lemma \ref{lem:Fujita}) and consider $\mathbf{L}_k: \PSH(X,\theta,\phi_k)\to \PSH(Y_k,\eta_k)$ be the map given by Lemma \ref{lem:Bijection}.
   We denote $v_{0,k}:=\mathbf{L}_k(u_{0,k})$,  $v_{1,k}:=\mathbf{L}_k(u_{1,k})$. Then by Lemma \ref{lem:Bijection} we know that $v_{t,k}:=\mathbf{L}_k(u_{t,k})$ is the psh geodesic segment joining $v_{0,k}$ and $v_{1,k}$. \\
   Since, for any $t\in [0,1)$, the sets   $\mathrm{Exc}(\pi_k), \pi_k\big(\mathrm{Exc}(\pi_k)\big), \{u_{t,k} = - \infty \},  \{u_{t,k}  \circ \pi_k = - \infty \},$ are pluripolar we obtain that 
    $$\dot{u}^+_{t,k} \circ \pi_k =\dot{v}^+_{t,k}, \qquad {\textit{outside a pluripolar set}} $$ and
    \begin{equation}
        \label{eqn:p_Finsler}
        \int_X |\dot{u}_{t,k}^+|^p \, \theta_{u_{t,k}}^n=\int_{Y_k} | \dot{v}_{t,k}^+ |^p \, (\eta_{k}+dd^c v_{t,k})^n.
    \end{equation}
 The same identity holds for the left derivatives for any $t\in (0, 1]$.
    \smallskip

 By Lemma \ref{lem:Holo} we have $u_{0,k}\circ \pi_k, u_{1,k}\circ\pi_k\in \Ent(Y_k,\pi_k^*\theta)$. Hence, thanks to Lemma \ref{lem:Bijection}(v) we infer that $v_{0,k},v_{1,k}\in \Ent(Y_k,\eta_k)$.
    
    Now, \cite[Theorem 4.4]{DNL21} ensures that $\dot{v}_{t,k}^+=\dot{v}_{t,k}^-=:\dot{v}_{t,k}$ almost everywhere with respect to $(\eta_{k}+dd^c v_{t,k})^n$ and that $\int_{Y_k} |\dot{v}_{t,k}|^p \,  (\eta_{k}+dd^c v_{t,k})^n$ is constant in $t\in [0,1]$.
   We then deduce that $\dot{u}_{t,k}^+=\dot{u}_{t,k}^-$ almost everywhere with respect to $\theta_{u_{t,k}}^n$ and from \eqref{eqn:p_Finsler} that
    $$
    \int_X |\dot{u}_{t,k}|^p \,\theta_{u_{t,k}}^n
    $$
    is constant in $t\in[0,1]$. \\
On the other hand, by \cite[Lemma 3.1]{DNL21} we can conclude that $\dot{u}_{t}^+ = \dot{u}_{t}^-:=\dot{u}_{t}$ almost everywhere with respect to $\theta_{u_t}^n.$ The latter lemma is proved when the reference form is K\"ahler. Nevertheless, the arguments in \cite[Lemma 3.1]{DNL21} work in the big setting as well since:
\begin{itemize}
    \item by \cite[Theorem 2.8]{DTT23} we know that $u_0, u_1\in \mathcal{E}^1(X, \theta)$. Hence $u_t\geq P_\theta(u_0, u_1) \in \mathcal{E}^1(X, \theta)$ thanks to  \cite[Theorem 2.13]{DDL1};
    \item the Monge-Amp\`ere energy $E(\theta; \cdot , V_\theta)$ is linear along psh geodesic by Theorem \ref{prop: energy linear} and concave along affine paths \cite[Corollary 5.9]{DDL6};
    \item $\Ent(\vol(\theta)^{-1}\theta_{u_t}^n, \omega^n)$ is uniformly bounded in $t$ thanks to Corollary \ref{cor: bound entropy geo}.
   \end{itemize}
   
    Next, we use the sequence $u_{t,k}$ (which has uniformly bounded entropy and weakly converges to $\theta_{u_t}^n$ thanks to Lemma \ref{u,u_k} and Remark \ref{rk:conv_incr}) to implement the arguments in \cite[pages 14-15]{DNL21}. This gives that for $t\in [0,1]$
    \begin{equation}\label{eqn:Useful_Later}
        \int_X |\dot{u}_{t,k}|^p \,\theta_{u_{t,k}}^n \rightarrow  \int_X |\dot{u}_{t}|^p \,\theta_{u_{t}}^n
    \end{equation}
    as $k\rightarrow +\infty$. Since, by the above, the quantity at the left hand side is constant in $t$, so are the quantities at the right hand side. The proof is then complete.
\end{proof}

Repeating word by word the arguments in \cite[Theorem 5.1]{DNL21} we have the following result which generalizes \cite[Theorem 5.1]{DNL21} and the main theorem in \cite{DNT21}.

\begin{prop}\label{thm: MA contact big}
Assume \eqref{UK}.
Assume $u \in \PSH(X,\theta)$, $v\in \Ent(X,\theta)$ and $u\leq v$. 
Then 
 \[
 {\bf 1}_{\{u = v\}} \theta_u^n = {\bf 1}_{\{u=v\}} \theta_v^n. 
 \]
\end{prop}


Following again word by word the proof of \cite[Corollary 5.2]{DNL21} we obtain:

\begin{corollary}\label{thm: domination}
Assume \eqref{UK}. Assume $u \in \PSH(X,\theta)$, $v\in \Ent(X,\theta)$ and $u = v$ $\theta^n_v$-almost everywhere, then $u = v.$ 
\end{corollary}

\section{Geodesic Distance}\label{sec geo dis}
In \cite{Gup23b}, Gupta showed how $\cE^p(X,\theta)$ is naturally endowed with a complete distance $d_p$. We briefly recall how such distance is constructed. Let  $\phi_k$ be a monotone transcendental Fujita approximation of $\{\theta\}$ (Lemma \ref{lem:Fujita}) and consider $\mathbf{L}_k: \PSH(X,\theta,\phi_k)\to \PSH(Y_k,\eta_k)$ be the map given by Lemma \ref{lem:Bijection}. Then each $\cE^p(X,\theta,\phi_k)$ can be endowed with a metric $d_p$. More precisely,
\begin{equation}\label{def dist big}
d_p(u_0,u_1):=d_p\left(\mathbf{L}_k(u_0),\mathbf{L}_k(u_1)\right)
\end{equation}
for any $u_0,u_1\in \cE^p(X,\theta,\phi_k)$ where the $d_p$-distance on $\cE^p(Y_k,\eta_k)$ is defined in \cite{DNL20}. Moreover, as proved in \cite[Theorem 7.4]{Gup23b} the $d_p$ distance on $\cE^p(X,\theta)$ satisfies
\begin{equation*}
    d_p(u_0,u_1)=\lim_{k\rightarrow +\infty} d_p\left(P_\theta[\phi_k](u_0),P_\theta[\phi_k](u_1)\right)
\end{equation*}
for any $u_0,u_1\in\cE^p(X,\theta)$.

\medskip

We prove here that the $d_p$ distance on $\cE^p(X,\theta)$ has an explicit expression when the potentials have the same singularity type and finite entropy. The following result extends \cite[Theorem 1.2]{DNL21} to the case of big cohomology classes.
\begin{thm}\label{thm dist big}
Assume \eqref{UK}. Fix $p\geq 1$. Let $u_0,u_1\in \Ent(X,\theta)\cap \cE^p(X,\theta)$ and $u_t$ be the psh geodesic connecting $u_0$ and $u_1$. If $u_0-u_1$ is bounded then
    \begin{equation}\label{eqn:Distance_Formula}
        d_p^p(u_0,u_1)=\int_X | \dot{u}_t|^p \,\theta_{u_t}^n
    \end{equation}
    for any $t\in [0,1]$.
\end{thm}
\noindent We recall that the right term in \eqref{eqn:Distance_Formula} is independent of $0\leq t\leq 1$ thanks to Theorem \ref{thm:Geodesic_Formula2}.
\begin{proof}
 Using the same notations of above we set $u_{i,k}:=P_\theta[\phi_k](u_i)$ for $i=0,1$ and $v_{i,k}:=\mathbf{L}_k(u_{i,k})$. We also denote by $u_{t,k},v_{t,k}$ the psh geodesics joining $u_{0,k},u_{1,k}$ and $v_{0,k},v_{1,k}$ respectively. Combining \cite[Theorem 1.2]{DNL21} with \eqref{eqn:p_Finsler} we have
    $$
    d_p^p(u_{0,k},u_{1,k})=d_p^p(v_{0,k},v_{1,k})=\int_{Y_k}\lvert \dot{v}_{t,k}\rvert^p(\eta_k+dd^cv_{t,k})^n=\int_X \lvert \dot{u}_{t,k} \rvert^p\theta_{u_{t,k}}^n,
    $$
    where we observe that the first identity follows from the definition given in \eqref{def dist big}.
    It follows from \eqref{eqn:Useful_Later} that
    $$
    d_p^p(u_0,u_1)=\lim_{k\rightarrow +\infty} d_p^p(u_{0,k},u_{1,k})=\lim_{k\rightarrow +\infty} \int_X \lvert \dot{u}_{t,k} \rvert^p\theta_{u_{t,k}}^n=\int_X \lvert \dot{u}_t \rvert^p\theta_{u_t}^n.
    $$
 This concludes the proof.
\end{proof}

 \bibliographystyle{plain}
\bibliography{biblio.bib}
 \bigskip

\noindent{\sc IMJ-PRG, Sorbonne Universit\'e \& \\ DMA, École normale supérieure, Université PSL, CNRS, \\
 4 place Jussieu \& 45 Rue d'Ulm, 75005 Paris, France,\\
 \tt {eleonora.dinezza@imj-prg.fr, edinezza@dma.ens.fr}}
\bigskip
  
  \noindent{\sc Universit\`a di Roma TorVergata, \\
  Via della Ricerca Scientifica 1, 00133, Roma, Italy, \\
  \tt{trapani@mat.uniroma2.it}}

  \bigskip
  \noindent{\sc Chalmers University of technology,\\
  Chalmers tvärgata 3, 412 96 Gothenburg, Sweden, \\
  \tt{trusiani@chalmers.se}}

\end{document}